\documentclass[reqno,12pt]{amsart}

\usepackage{graphics}
\usepackage{amssymb, amsmath}
\usepackage{version}
\usepackage{fancybox}
\usepackage{bm}
\usepackage{color}
\usepackage{mathrsfs}
\usepackage{mathtools}
\theoremstyle{theorem}
\newtheorem{theorem}{\sc Theorem}[section]
\newtheorem{lemma}[theorem]{\sc Lemma}

\newtheorem{proposition}[theorem]{\sc Proposition}
\newtheorem{corollary}[theorem]{\sc Corollary}

\theoremstyle{definition}
\newtheorem{definition}[theorem]{\sc Definition}

\theoremstyle{remark}
\newtheorem{remark}[theorem]{\sc Remark}
\makeatletter

\@addtoreset{equation}{section}

\renewcommand{\d}{\text{\rm d}}
\newcommand{\D}{\text{\rm D}}
\newcommand{\vep}{\varepsilon}

\newcommand{\e}{{\rm e}}

\newcommand{\R}{\mathbb{R}}
\newcommand{\N}{\mathbb{N}}
\newcommand{\Linv}{(-\Delta)^{-1}}

\newcommand{\lamq}{\lambda_q}
\newcommand{\lamo}{\lambda_0}

\newcommand{\vm}{|v|^{q-2}v}
\newcommand{\vmh}{|v|^{(q-2)/2}v}

\newcommand{\Lphi}{\mathcal L_\phi}

\allowdisplaybreaks[4]

\begin{document}
\title[Rates of convergence for fast diffusion]{Rates of convergence\\to non-degenerate asymptotic profiles\\for fast diffusion via energy methods}
%
%
\author{Goro Akagi}
\address{Mathematical Institute and Graduate School of Science, Tohoku University, Aoba, Sendai 980-8578, Japan}
\email{goro.akagi@tohoku.ac.jp}
%
%
\maketitle
\begin{abstract}
This paper is concerned with a quantitative analysis of asymptotic behaviors of (possibly sign-changing) solutions to the Cauchy-Dirichlet problem for the fast diffusion equation posed on bounded domains with Sobolev subcritical exponents. More precisely, rates of convergence to non-degenerate asymptotic profiles will be revealed via an energy method. The sharp rate of convergence to \emph{positive} ones was recently discussed by Bonforte and Figalli~\cite{BF21} based on an entropy method. An alternative proof for their result will also be provided. Furthermore, dynamics of fast diffusion flows with changing signs will be discussed more specifically under concrete settings; in particular, exponential stability of some sign-changing asymptotic profiles will be proved in dumbbell domains for initial data with certain symmetry.
\end{abstract}

\section{Introduction}\label{S:Intro}

Let $\Omega$ be a bounded $C^{1,1}$ domain of $\mathbb R^N$ with boundary $\partial \Omega$. We are concerned with the Cauchy-Dirichlet problem
for the fast diffusion equation of the form,
\begin{alignat}{4}
\partial_t \left( |u|^{q-2}u \right) &= \Delta u
 \quad && \mbox{ in } \Omega \times (0, \infty),\label{eq:1.1}\\
 u &= 0 && \mbox{ on } \partial \Omega \times (0, \infty), \label{eq:1.2}\\
 u &= u_0 && \mbox{ in } \Omega \times \{0\},\label{eq:1.3}
\end{alignat}
where $\partial_t = \partial/\partial t$, under the assumptions that
\begin{equation*}
u_0 \in H^1_0(\Omega), \quad 2 < q < 2^* := \dfrac{2N}{(N-2)_+}.
\end{equation*}
The Cauchy-Dirichlet problem \eqref{eq:1.1}--\eqref{eq:1.3} arises from the Okuda-Dawson model (see~\cite{Okuda-Dawson}), which describes an anomalous diffusion of plasma (see also~\cite{BH78,BH82}). We refer the reader to~\cite[\S 2]{A16} for the definition of \emph{weak solutions} concerned in the present paper and their existence and regularity along with a couple of energy estimates (see also~\cite{Vaz06,Vazquez} as a general reference).

It is well known that every weak solution $u = u(x,t)$ of \eqref{eq:1.1}--\eqref{eq:1.3} vanishes at a finite time $t_*$, which is uniquely determined by the initial datum $u_0$ (see~\cite{Sabinina62,BC,Diaz88,HerreroVazquez88}); hence, we may write $t_* = t_*(u_0)$. Moreover, Berryman and Holland~\cite{BH80} proved that the rate of finite-time extinction of $u(\cdot,t)$ is just $(t_*-t)_+^{1/(q-2)}$ as $t \nearrow t_*$, that is,
$$
c_1 (t_* - t)_+^{1/(q-2)} \leq \|u(\cdot,t)\|_{H^1_0(\Omega)} 
\leq c_2 (t_* - t)_+^{1/(q-2)}
\quad  \mbox{ for all } \ t \geq 0
$$
with $c_1, c_2 > 0$, provided that $u_0 \not\equiv 0$ (see also~\cite{Kwong88-3,DKV91,SavareVespri,BGV08,BV10}). Then we define the asymptotic profile $\phi(x)$ of $u(x,t)$ as 
\begin{equation}\label{ap}
\phi(x) = \lim_{t\nearrow t_*} (t_*-t)^{-1/(q-2)} u(x,t) \not\equiv 0 \ \mbox{ in } H^1_0(\Omega).
\end{equation}
Apply the change of variables, 
\begin{equation}\label{cv}
v(x,s) = (t_* - t)^{-1/(q-2)} u(x,t)
\quad \mbox{ with } \  s = \log (t_*/(t_* - t)).
\end{equation}
Then $v=v(x,s)$ solves the following rescaled problem:
\begin{alignat}{4}
 \partial_s \left( |v|^{q-2}v \right) &= \Delta v + \lamq |v|^{q-2}v
\quad && \mbox{ in } \Omega \times (0, \infty),\label{eq:1.6}\\
 v &= 0 && \mbox{ on } \partial \Omega \times (0, \infty), \label{eq:1.7}\\
 v &= v_0 && \mbox{ in } \Omega \times \{0\}\label{eq:1.8}
\end{alignat}
with $\lamq := (q-1)/(q-2) > 0$ and the initial datum 
\begin{equation}\label{v0}
v_0 := t_*(u_0)^{-1/(q-2)}u_0.
\end{equation}
Here it is worth mentioning that such rescaled initial data form the set 
\begin{align}
\mathcal{X} &:= \{ t_*(u_0)^{-1/(q-2)}u_0 \colon u_0 \in H^1_0(\Omega) \setminus \{0\}\} \label{phase_set}\\
&= \{w \in H^1_0(\Omega) \colon t_*(w) = 1\}\nonumber
\end{align}
(see~\cite[Proposition 6]{AK13} for the equality) and it plays a role of the phase set in stability analysis of asymptotic profiles (see Definition \ref{D:stbl} below and~\cite{AK13} for more details). Now, the asymptotic profile $\phi(x)$ is reformulated as the limit of $v(x,s)$ as $s \to \infty$; moreover, profiles are characterized as nontrivial solutions to the stationary problem,
\begin{alignat}{4}
 - \Delta \phi &= \lamq |\phi|^{q-2}\phi \quad && \mbox{ in } \Omega,
\label{eq:1.10}\\
\phi &= 0 && \mbox{ on } \partial \Omega,
\label{eq:1.11}
\end{alignat}
and vice versa. On the other hand, although \emph{quasi-convergence} (i.e., convergence along a subsequence) of $v(\cdot,s)$ follows from a standard argument (see, e.g.,~\cite{BH80,Kwong88-3,DKV91,SavareVespri,AK13}), \emph{convergence} (along the whole sequence) is more delicate. Actually, it is proved in~\cite{FeiSim00} for non-negative bounded solutions with the aid of \L ojasiewicz-Simon's gradient inequality; however, it still seems open for possibly sign-changing solutions, unless asymptotic profiles are isolated in $H^1_0(\Omega)$ or $q$ is even (i.e., analytic nonlinearity). Moreover, in~\cite{BGV12}, convergence of relative errors for non-negative solutions is also proved, that is,
\begin{equation}\label{relerr-conv}
\lim_{t \nearrow t_*} \left\| \frac{u(\cdot,t)}{(t_*-t)^{1/(q-2)}\phi}-1\right\|_{C(\overline\Omega)} = \lim_{s \to \infty} \left\| \frac{v(\cdot,s)}{\phi}-1\right\|_{C(\overline\Omega)} = 0.
\end{equation}
Furthermore, rates of convergence are discussed in~\cite{BGV12}, where an exponential convergence of the so-called relative entropy (see Corollary \ref{C:entropy} below) was first proved; however, it seems still rather difficult to quantitatively estimate the rate of convergence. The \emph{sharp rate} (see below) of convergence for \emph{non-degenerate} (see below) positive asymptotic profiles was first discussed in~\cite{BF21} by developing the so-called \emph{nonlinear entropy method}. We also refer the reader to recent developments~\cite{JiXi01,JiXi02}.

Throughout this paper, as in~\cite{BF21}, we assume that $\phi$ is \emph{non-degenerate}, i.e., the linearized problem
$$
\Lphi(u) := -\Delta u - \lambda_q (q-1) |\phi|^{q-2} u = 0
$$
admits no non-trivial solution (or equivalently, $\Lphi$ does not have zero eigenvalue), and hence, $\Lphi$ is invertible. Then $\phi$ is also isolated in $H^1_0(\Omega)$ from the other solutions to \eqref{eq:1.10}, \eqref{eq:1.11}, that is, there exists a neighbourhood of $\phi$ in $H^1_0(\Omega)$ which does not involve any other solutions to \eqref{eq:1.10}, \eqref{eq:1.11}. We shall denote by $\{\mu_j\}_{j=1}^\infty$ the non-decreasing sequence consisting of all the eigenvalues for the eigenvalue problem,
\begin{equation}\label{wep}
 -\Delta e = \mu |\phi|^{q-2} e \ \mbox{ in } \Omega, \quad e = 0 \ \mbox{ on } \partial \Omega. 
\end{equation}
Then thanks to the spectral theory for compact self-adjoint operators (see, e.g.,~\cite{B-FA}), we find that $0 < \mu_1 < \mu_2 \leq \cdots \leq \mu_j \to +\infty$ as $j \to +\infty$. Moreover, the eigenfunctions $\{e_j\}_{j=1}^\infty$ form a complete orthonormal system (CONS for short) in $H^1_0(\Omega)$ and also a CONS in a weighted $L^2$ space $L^2(\Omega;|\phi|^{q-2}\d x)$ with different normalization. As for positive profiles $\phi$, a slightly different form of the eigenvalue problem \eqref{wep} has already been employed in~\cite{BF21} (see also Remark \ref{R:BF21} below).

As in~\cite[\S 2]{BF21}, the \emph{sharp rate} of convergence is defined for non-degenerate positive asymptotic profiles $\phi>0$ in view of a linearized analysis of \eqref{eq:1.6}--\eqref{eq:1.8}. More precisely, we consider the (formally) linearized equation (i.e., linearization of \eqref{eq:1.6}--\eqref{eq:1.8} at $\phi$),
\begin{alignat*}{4}
(q-1) \phi^{q-2} \partial_s h &= \Delta h + \lambda_q (q-1) \phi^{q-2} h \quad &&\mbox{ in } \Omega \times (0,\infty),\\
h &\,=\, 0 \ &&\mbox{ on } \partial \Omega \times (0,\infty),\\
h(\cdot,0) &\,=\, h_0 := v_0-\phi \ &&\mbox{ in } \Omega,
\end{alignat*}
where the solution $h = h(x,s)$ may correspond to the difference between $v(x,s)$ and $\phi(x)$. Then for a certain class of initial data $h_0$ the (linear) entropy
$$
\mathsf{E}[h(s)] = \int_\Omega h(x,s)^2 \phi(x)^{q-2} \, \d x
$$
turns out to decay at the exponential rate $\e^{-\lamo s}$ with the exponent
\begin{equation}\label{lambda}
\lamo = \frac 2 {q-1} \left[ \mu_k - \lambda_q(q-1) \right] > 0,
\end{equation}
where $k \in \N$ is the least integer, i.e., $\mu_k$ is the least eigenvalue for \eqref{wep}, such that $\mu_k > \lambda_q(q-1)$ (that is, $\nu_k := \mu_k - \lambda_q (q-1)$ is the least positive eigenvalue of $\Lphi$). Here and henceforth, the convergence rate mentioned above (or the exponent $\lambda_0$ as in \eqref{lambda}) is called a \emph{sharp rate}. In contrast with the porous medium equation (i.e., the case for $1 < q < 2$), which is studied in~\cite{ArPel81} by comparison arguments (see also~\cite{BSV15,BFV18,Vaz04} based on Global Harnack principle or entropy methods and~\cite[Theorem 3.4]{BGV12}, where an entropy method is developed for the PME), it is more difficult to directly prove the optimality of the convergence rate for \eqref{eq:1.6}--\eqref{eq:1.8} due to the nature of finite-time extinction phenomena of solutions for the fast diffusion equation. To be more precise, the major difficulty consists in comparing solutions with barriers near the extinction time; in particular, it is rather difficult to construct sub- and supersolutions that vanish at the same time as the solutions.

Define the energy functional $J : H^1_0(\Omega) \to \R$ by
\begin{equation}\label{J}
J(w) := \frac 12 \int_\Omega |\nabla w(x)|^2 \, \d x - \frac{\lambda_q}{q} \int_\Omega |w(x)|^q \, \d x
\end{equation}
for $w \in H^1_0(\Omega)$. We are ready to state main results of the present paper. 

\begin{theorem}[Convergence with rates to sign-changing profiles]\label{T:sc-conv}
Let $v = v(x,s)$ be a {\rm(}possibly sign-changing{\rm)} weak solution to \eqref{eq:1.6}--\eqref{eq:1.8} and let $\phi = \phi(x)$ be a {\rm(}possibly sign-changing{\rm)} nontrivial solution to \eqref{eq:1.10}, \eqref{eq:1.11} such that $v(\cdot,s_n) \to \phi$ strongly in $H^1_0(\Omega)$ for some $s_n \to +\infty$. Suppose that $\phi$ is non-degenerate. Let $\lambda$ be a constant satisfying
\begin{equation}\label{exp-est}
0 < \lambda < \frac 2{q-1} C_q^{-2} \|\phi\|_{L^q(\Omega)}^{-(q-2)} \frac{\mu_k - \lambda_q(q-1)}{\mu_k},
\end{equation}
where $\mu_k$ is the least eigenvalue for \eqref{wep} greater than $\lambda_q(q-1)$ and $C_q$ is the best constant of the Sobolev-Poincar\'e inequality,
\begin{equation}\label{SI}
\|w\|_{L^q(\Omega)} \leq C_q \|\nabla w\|_{L^2(\Omega)} \quad \mbox{ for } \ w \in H^1_0(\Omega).
\end{equation}
Then there exists a constant $C_\lambda > 0$ depending on the choice of $\lambda$ such that
\begin{equation}\label{c:J-exp:sc}
0 \leq J(v(s)) - J(\phi) \leq C_\lambda \e^{-\lambda s} \quad \mbox{ for } \ s \geq 0.
\end{equation}
Moreover, there exists a constant $M_\lambda > 0$ depending on the choice of $\lambda$ such that
$$
\|v(s)-\phi\|_{H^1_0(\Omega)}^2 \leq M_\lambda \e^{-\lambda s} \quad \mbox{ for } \ s \geq 0.
$$
\end{theorem}

It is noteworthy that Theorem \ref{T:sc-conv} is concerned with \emph{possibly sign-changing} weak solutions to \eqref{eq:1.6}--\eqref{eq:1.8} and their limits, i.e., nontrivial solutions to \eqref{eq:1.10}, \eqref{eq:1.11}. It is well known that \eqref{eq:1.10}, \eqref{eq:1.11} admits infinitely many sign-changing solutions in general (see, e.g.,~\cite{Rabinowitz}). Moreover, in Section \ref{S:sc-beh}, we shall exhibit several examples of sign-changing initial data $u_0$ and domains $\Omega$ for which the (sign-changing) weak solutions $u = u(x,t)$ to \eqref{eq:1.1}--\eqref{eq:1.3} admit sign-definite and sign-changing asymptotic profiles, although sign-changing asymptotic profiles are often unstable (see~\cite{AK13}).

As a by-product of the theorem above, we can also prove \emph{exponential stability} of non-degenerate asymptotic profiles which takes the least energy among all the profiles. Let us first recall the notion of stability and instability of asymptotic profiles for fast diffusion, which was introduced in~\cite{AK13} (see also~\cite{INdAM13,AK14,A16}) and will also be used in \S \ref{S:sc-beh}. Here $\mathcal{X}$ is the phase set defined in \eqref{phase_set}.

\begin{definition}[Stability and instability of asymptotic profiles (cf.~\cite{AK13})]
\label{D:stbl}
{\ }Let $\phi$ be an asymptotic profile of a weak solution to \eqref{eq:1.1}--\eqref{eq:1.3} {\rm (}equivalently, a nontrivial solution to \eqref{eq:1.10}, \eqref{eq:1.11}{\rm )}.
\begin{enumerate}
\item $\phi$ is said to be \emph{stable}, if
      for any $\varepsilon>0$ there exists $\delta > 0$ such
      that any solution $v$ of \eqref{eq:1.6}, \eqref{eq:1.7} satisfies
	       \begin{equation*}
		\sup_{s \in [0, \infty)}
		 \|v(s) -\phi\|_{H^1_0(\Omega)} < \varepsilon,
	       \end{equation*}
	     whenever $v(0) \in \mathcal X$ and $\|v(0) - \phi\|_{H^1_0(\Omega)}<\delta$.
\item $\phi$ is said to be \emph{unstable}, if $\phi$ is not stable.
\item $\phi$ is said to be \emph{asymptotically stable}, if $\phi$ is stable,
      and moreover, there exists $\delta_0 > 0$
      such that any solution $v$ of \eqref{eq:1.6}, \eqref{eq:1.7}
      satisfies
      \begin{equation*}
       \lim_{s \nearrow \infty}\|v(s) - \phi\|_{H^1_0(\Omega)} = 0,
      \end{equation*}
      whenever $v(0) \in \mathcal X$ and $\|v(0) - \phi\|_{H^1_0(\Omega)}<\delta_0$.
\item[(iv)] $\phi$ is said to be \emph{exponentially stable}, if $\phi$ is stable, and moreover, there exist constants $C, \mu, \delta_1 > 0$ such that any solution $v$ of \eqref{eq:1.6}, \eqref{eq:1.7} satisfies
$$
\|v(s)-\phi\|_{H^1_0(\Omega)} \leq C \e^{-\mu s} \quad \mbox{ for all } \ s \geq 0,
$$
provided that $v(0) \in \mathcal X$ and $\|v(0)-\phi\|_{H^1_0(\Omega)} < \delta_1$.
\end{enumerate}
\end{definition}

In what follows, the \emph{least-energy solutions} to \eqref{eq:1.10}, \eqref{eq:1.11} (or \emph{least-energy asymptotic profiles}) mean nontrivial solutions to \eqref{eq:1.10}, \eqref{eq:1.11} minimizing the energy $J$ among all the nontrivial solutions to \eqref{eq:1.10}, \eqref{eq:1.11}.

\begin{corollary}[Exponential stability of non-degenerate least-energy profiles]\label{C:exp-stbl}
{\ }Non-degenerate least-energy asymptotic profiles $\phi$ are exponentially stable in the sense of Definition \ref{D:stbl}. In particular, for any $\lambda$ satisfying \eqref{exp-est}, there exist constants $C, \delta_0 > 0$ such that any solution $v = v(x,s)$ of \eqref{eq:1.6}--\eqref{eq:1.8} satisfies
$$
\|v(s)-\phi\|_{H^1_0(\Omega)} \leq C \e^{-\lambda s/2} \quad \mbox{ for all } \ s \geq 0,
$$
provided that $v(0) \in \mathcal X$ and $\|v(0)-\phi\|_{H^1_0(\Omega)} < \delta_0$.
\end{corollary}

If we restrict ourselves to \emph{non-negative} weak solutions, we can derive more precise results. 

\begin{theorem}[Sharp convergence rate of energy]\label{T:main}
Let $v = v(x,s)$ be a non-negative weak solution of \eqref{eq:1.6}--\eqref{eq:1.8} and let $\phi$ be a positive solution to \eqref{eq:1.10}, \eqref{eq:1.11} such that $v(s_n) \to \phi$ strongly in $H^1_0(\Omega)$ for some $s_n \to +\infty$. Assume that $\phi$ is non-degenerate. Then there exists a constant $C > 0$ such that
\begin{equation}\label{sharp-rate}
0 \leq J(v(s))-J(\phi) \leq C \e^{-\lamo s} \quad \mbox{ for } \ s \geq 0, 
\end{equation}
where $\lamo >0$ is given as in \eqref{lambda}. 
\end{theorem}

The rate of convergence in \eqref{sharp-rate} is faster than \eqref{c:J-exp:sc} obtained in Theorem \ref{T:sc-conv} for (possibly) sign-changing solutions (see Remark \ref{R:leap} below). The preceding theorem yields the following corollary, which provides an alternative proof for~\cite[Theorem 1.2]{BF21}:

\begin{corollary}[Sharp convergence rate of relative entropy]\label{C:entropy}
Under the same assumptions as in Theorem \ref{T:main}, there exists a constant $C > 0$ such that
\begin{equation}\label{ent-conv}
\int_\Omega |v(x,s)-\phi(x)|^2 \phi(x)^{q-2} \, \d x \leq C \e^{-\lamo s} \quad \mbox{ for } \ s \geq 0,
\end{equation}
where $\lamo$ is given as in \eqref{lambda}.
\end{corollary}

Thanks to the energy convergence (along with the entropic one), we can also derive the sharp convergence rate of the $H^1_0$-norm.

\begin{corollary}[Sharp convergence rate of $H^1_0$-norm]\label{C:H10}
Under the same assumptions as in Theorem \ref{T:main}, there exists a constant $C > 0$ such that
\begin{equation}\label{H10-conv}
\int_\Omega |\nabla v(x,s)- \nabla \phi(x)|^2  \, \d x \leq C \e^{-\lamo s} \quad \mbox{ for } \ s \geq 0,
\end{equation}
where $\lamo$ is given as in \eqref{lambda}. Moreover, it also holds that
\begin{equation}\label{J'-conv}
\left\| \partial_s \left( v^{q-1}\right)(s)\right\|_{H^{-1}(\Omega)}
= \left\| J'(v(s)) \right\|_{H^{-1}(\Omega)} \leq C \e^{-\frac\lamo 2 s}
\end{equation}
for $s \geq 0$.
\end{corollary}

Scaling back to the original variable, we can readily rewrite Corollaries \ref{C:entropy} and \ref{C:H10} as follows:
\begin{corollary}
Let $u = u(x,t)$ be a non-negative weak solution of \eqref{eq:1.1}--\eqref{eq:1.3} with a finite extinction time $t_* > 0$ and let $\phi$ be a positive solution to \eqref{eq:1.10}, \eqref{eq:1.11} such that $(t_* - t)^{-1/(q-2)} u(t) \to \phi$ strongly in $H^1_0(\Omega)$ as $t \nearrow t_*$. Assume that $\phi$ is non-degenerate. Then there exists a constant $C > 0$ such that
\begin{align}
\int_\Omega \left| \frac{u(x,t)}{(t_*-t)^{1/(q-2)}\phi(x)} - 1\right|^2 \phi(x)^q \, \d x &\leq C \left(\frac{t_* - t}{t_*}\right)^{\lamo},\label{rel_err_conv_u}\\
\int_\Omega \left| (t_*-t)^{-1/(q-2)}\nabla u(x,t) - \nabla \phi(x) \right|^2 \, \d x &\leq C \left(\frac{t_* - t}{t_*}\right)^{\lamo},\label{H10_conv_u}
\end{align}
where $\lamo$ is given as in \eqref{lambda}, for $t \in [0,t_*)$.
\end{corollary} 

The topology of convergence (with the sharp rate) in Corollary \ref{C:H10} seems slightly stronger than the main theorem of~\cite{BF21} (see Remark \ref{R:BF21} below); however, with the aid of a recent boundary regularity result (for \emph{non-negative} solutions on \emph{smooth} domains) established by~\cite{JiXi01}, convergences with the sharp rate in stronger topologies also follow from the relative error convergence in the weighted $L^2$ space obtained in~\cite{BF21} (see Corollary \ref{C:entropy}). On the other hand, the main results of the present paper will be proved in a different way, which relies on an \emph{energy method} rather than the entropy method and which may be much simpler than the method used in~\cite{BF21}. In particular, we can avoid the argument to prove some improvement of the ``almost orthogonality'' along the nonlinear flow (see \S 3.2-3.6 of~\cite{BF21}), which may be the most involved part of the paper~\cite{BF21}. Furthermore, it is also noteworthy that all the main results of the present paper can be proved for arbitrary bounded $C^{1,1}$ domains (see Remark \ref{R:C11} below for details). 

\begin{remark}[Comparison with~\cite{BF21}]\label{R:BF21}
Throughout this paper, we shall use the transformations \eqref{cv}, which are slightly different from those used in~\cite{BF21}. Moreover,~\cite{BF21} is concerned with an eigenvalue problem, which is also slightly different from \eqref{wep} and whose eigenvalues $\lambda_{V,k}$, $k \geq 1$ coincide with $\mu_j/t_*$ of the present paper for $\sum_{\ell=1}^{k-1} N_\ell < j \leq \sum_{\ell=1}^{k} N_\ell$ (here $N_\ell$ denotes the dimension of the $\ell$-th eigenspace), since the profile function $V$ used in~\cite{BF21} corresponds to $t_*^{1/(q-2)}\phi$ of ours. On the other hand, the sharp rate $\lamo$ as in \eqref{lambda} coincides with $2T \lambda_m$ as in~\cite{BF21} with $T = t_*$; hence, \eqref{ent-conv} and \eqref{rel_err_conv_u} are completely same as the assertion of~\cite[(1.15) of Theorem 1.2 and (1.18) of Remark 1.3]{BF21}.
\end{remark}

\bigskip
\noindent
{\bf Plan of the paper.} Sections \ref{S:csc}--\ref{S:exp-conv} are devoted to a proof for Theorem \ref{T:sc-conv}. Sections \ref{S:almost_sharp}--\ref{S:2<q<3} are concerned with a proof for Theorem \ref{T:main}. In Section \ref{S:Cor}, Corollaries \ref{C:exp-stbl}, \ref{C:entropy} and \ref{C:H10} will be proved. In Section \ref{S:sc-beh}, fast diffusion flows with changing signs are discussed; in particular, exponential stability of some sign-changing asymptotic profiles will be proved in dumbbell domains for initial data with certain symmetry. In Appendix, we shall recall Taylor's theorem for operators in Banach spaces as well as some fundamental inequalities. 

\bigskip
\noindent
{\bf Notation.} We denote by $C$ a generic non-negative constant which may vary from line to line. Moreover, $q' := q/(q-1)$ denotes the H\"older conjugate of $q \in (1,\infty)$. Furthermore, denote by $H^{-1}(\Omega)$ the dual space of the Sobolev space $H^1_0(\Omega)$ equipped with the inner product $(u,v)_{H^1_0(\Omega)} = \int_\Omega \nabla u \cdot \nabla v \, \d x$ for $u,v \in H^1_0(\Omega)$. Moreover, an inner product of $H^{-1}(\Omega)$ is naturally defined as
\begin{equation}\label{H-1product}
(f,g)_{H^{-1}(\Omega)} = \langle f, (-\Delta)^{-1}g\rangle_{H^1_0(\Omega)} \quad \mbox{ for } \ f,g \in H^{-1}(\Omega),
\end{equation}
which also gives $\|f\|_{H^{-1}(\Omega)}^2 = (f,f)_{H^{-1}(\Omega)}$ for $f \in H^{-1}(\Omega)$. Then $-\Delta$ is a duality mapping between $H^1_0(\Omega)$ and $H^{-1}(\Omega)$, that is,
\begin{align*}
\|u\|_{H^1_0(\Omega)}^2 &= \|-\Delta u\|_{H^{-1}(\Omega)}^2 = \langle -\Delta u, u \rangle_{H^1_0(\Omega)},\\
\|f\|_{H^{-1}(\Omega)}^2 &= \|(-\Delta)^{-1} f\|_{H^1_0(\Omega)}^2 = \langle f, (-\Delta)^{-1} f \rangle_{H^1_0(\Omega)}
\end{align*}
for $u \in H^1_0(\Omega)$ and $f \in H^{-1}(\Omega)$. Let $X$ and $Y$ be Banach spaces and denote by $\mathscr{L}^{(n)}(X,Y)$ the set of all bounded $n$-linear forms from $X$ into $Y$ for $n \in \N$ (in particular, $\mathscr{L}(X,Y) = \mathscr{L}^{(1)}(X,Y)$). In particular, we write $\mathscr{L}(X) = \mathscr{L}(X,X)$. Let $T : X \to Y$ be an operator. We denote by $\D_G T$ the G\^ateaux derivative of $T$. Moreover, the $n$-th Fr\'echet derivative of $T$ is denoted by $T^{(n)}$ for $n \in \N$ (we shall write $T' = T^{(1)}$ and $T'' = T^{(2)}$ for short). 

\section{Convergence with rates for possibly sign-changing asymptotic profiles}\label{S:csc}

Through the following three sections, we shall give a proof for Theorem \ref{T:sc-conv}. Let $v = v(x,s)$ be a (possibly sign-changing) weak solution to \eqref{eq:1.6}--\eqref{eq:1.8} and let $\phi = \phi(x)$ be a non-degenerate (possibly sign-changing) solution to \eqref{eq:1.10}, \eqref{eq:1.11} such that $v(s_n) \to \phi$ strongly in $H^1_0(\Omega)$ for some $s_n \to +\infty$. Then we first claim that
\begin{equation}\label{ap-conv}
v(s) \to \phi \quad \mbox{ strongly in } H^1_0(\Omega) \ \mbox{ as } \ s \to +\infty.
\end{equation}
Indeed, it is well known that every non-degenerate solution $\phi$ is isolated in $H^1_0(\Omega)$ (see, e.g.,~\cite[\S 5.3]{AK13}), that is, there exists $r > 0$ such that the ball $B_{H^1_0(\Omega)}(\phi;r) = \{w \in H^1_0(\Omega) \colon \|w - \phi\|_{H^1_0(\Omega)} < r\}$ does not involve any solutions to \eqref{eq:1.10}, \eqref{eq:1.11} except for $\phi$. Now, suppose to the contrary that there exist a sequence $\sigma_n \to +\infty$ and a constant $r_0 > 0$ such that $\|v(\sigma_n) - \phi\|_{H^1_0(\Omega)} > r_0$ for any $n \in \N$. Then due to~\cite[Theorem 1]{AK13}, up to a (not relabeled) subsequence, $v(\sigma_n) \to \psi$ strongly in $H^1_0(\Omega)$ for another (nontrivial) solution $\psi$ to \eqref{eq:1.10}, \eqref{eq:1.11}. Then since $\|\phi - \psi\|_{H^1_0(\Omega)} \geq r$, one can take a sequence $\tilde{s}_n \to +\infty$ such that $\|v(\tilde{s}_n) - \phi\|_{H^1_0(\Omega)} = r/2$ (cf.~see~\cite[Proof of Theorem 3]{A16}). However, one can take a (not relabeled) subsequence of $(\tilde{s}_n)$ such that $v(\tilde{s}_n) \to \tilde{\phi}$ strongly in $H^1_0(\Omega)$ for some nontrivial solution $\tilde\phi$ to \eqref{eq:1.10}, \eqref{eq:1.11} and $\|\tilde\phi-\phi\|_{H^1_0(\Omega)} = r/2$. It is a contradiction. Thus \eqref{ap-conv} follows. Moreover, we can assume $v(s) \neq \phi$ for any $s>0$; otherwise, $v(s) \equiv \phi$ for any $s > 0$ large enough.

Formally test \eqref{eq:1.6} by $\partial_s v(s)$ to see that
\begin{equation}\label{en}
\frac 4 {qq'} \left\| \partial_s (\vmh)(s) \right\|_{L^2(\Omega)}^2 \leq - \dfrac{\d}{\d s} J(v(s)),
\end{equation}
where $J: H^1_0(\Omega) \to \R$ is the functional given by \eqref{J} (this procedure can be justified via construction of weak solutions and their uniqueness; see, e.g.,~\cite{G:EnSol} and also~\cite{BII} for the fractional case, cf.~\cite{HB3}). Noting that
\begin{equation}\label{pvm-exp}
\partial_s (\vm)(s) = \frac{2(q-1)}q |v(s)|^{(q-2)/2} \partial_s (\vmh)(s),
\end{equation}
we also find from \eqref{ap-conv} along with the embedding $H^1_0(\Omega) \hookrightarrow L^q(\Omega)$ that, for any $\vep > 0$, there exists $s_\vep > 0$ large enough such that
\begin{align*}
\MoveEqLeft{
\left\| \partial_s (\vm) (s)\right\|_{H^{-1}(\Omega)}
}\\
&\leq C_q \left\| \partial_s (\vm) (s)\right\|_{L^{q'}(\Omega)}\\
&\leq \frac{2(q-1)}q C_q \|v(s)\|_{L^q(\Omega)}^{(q-2)/2} \left\| \partial_s (\vmh)(s) \right\|_{L^2(\Omega)}\\
&\leq \frac{2(q-1)}q C_q \left( \|\phi\|_{L^q(\Omega)}+\vep \right)^{(q-2)/2} \left\| \partial_s (\vmh)(s) \right\|_{L^2(\Omega)}
\end{align*}
for all $s \geq s_\vep$. Here $C_q$ denotes the best constant of the Sobolev-Poincar\'e inequality \eqref{SI}. As above, we shall often use the dual inequality of \eqref{SI},
\begin{equation}\label{d-SP}
 \|f\|_{H^{-1}(\Omega)} \leq C_q \|f\|_{L^{q'}(\Omega)} \quad \mbox{ for } \ f \in L^{q'}(\Omega),
\end{equation}
which is equivalent to \eqref{SI} by duality. Hence $C_q$ is also best for \eqref{d-SP} (see also~\cite[Appendix 7.8]{BSV15} and~\cite{BDNS}). Combining the above with \eqref{en}, we infer that
\begin{align}\label{2}
\frac{1}{q-1} C_q^{-2} \left( \|\phi\|_{L^q(\Omega)}+\vep \right)^{-(q-2)} \left\| \partial_s (\vm) (s)\right\|_{H^{-1}(\Omega)}^2\quad\nonumber\\
\leq - \frac{\d}{\d s} J(v(s)) \quad \mbox{ for } \ s \geq s_\vep.
\end{align}

We shall next derive the following gradient inequality: 
\begin{lemma}[Gradient inequality]\label{L:GI}
For any constant $\omega > Q_{\phi}^{1/2}/\sqrt{2}$, where
$$
Q_\phi := \sup \left\{ \left\langle h, \Lphi^{-1}(h) \right\rangle_{H^1_0(\Omega)} \colon h \in H^{-1}(\Omega), \ \|h\|_{H^{-1}(\Omega)}=1 \right\} > 0,
$$
there exists a constant $\delta > 0$ such that
\begin{align}\label{grad-ineq}
 \left( J(w)-J(\phi) \right)_+^{1/2} \leq \omega \|J'(w)\|_{H^{-1}(\Omega)} \quad \mbox{ for } \ w \in H^1_0(\Omega),
\end{align}
provided that $\|w-\phi\|_{H^1_0(\Omega)} < \delta$. 
\end{lemma}

\begin{proof}
As $J$ is of class $C^2$ in $H^1_0(\Omega)$, by Taylor's theorem (see Theorem \ref{T:Taylor} and Remark \ref{R:Taylor} in Appendix), one finds that
 \begin{equation}\label{T-1}
  J(\phi + h) = J(\phi) + \frac 12 \langle \Lphi h, h \rangle_{H^1_0(\Omega)} + R(h) \quad \mbox{ for } \ h \in H^1_0(\Omega),
 \end{equation}
 where we used the fact that $J'(\phi) = 0$ and $R(\cdot)$ denotes a generic functional defined on $H^1_0(\Omega)$ satisfying
 \begin{equation}\label{a:e}
 \lim_{\|h\|_{H^1_0(\Omega)} \to 0} \dfrac{|R(h)|}{\|h\|_{H^1_0(\Omega)}^2} = 0
 \end{equation}
and may vary from line to line. Moreover, one can take an operator $r : H^1_0(\Omega) \to H^{-1}(\Omega)$ such that
 \begin{equation}\label{T-2}
  J'(\phi+h) = \Lphi h + r(h) \ \mbox{ in } H^{-1}(\Omega)
   \quad \mbox{ for } \ h \in H^1_0(\Omega)
 \end{equation}
 and
 \begin{equation}\label{a:E}
 \lim_{\|h\|_{H^1_0(\Omega)}\to 0} \dfrac{\|r(h)\|_{H^{-1}(\Omega)}}{\|h\|_{H^1_0(\Omega)}} = 0. 
\end{equation}
Hence it follows that
  \begin{align}
   \MoveEqLeft{
   J(w)-J(\phi)
   }\nonumber\\
   &\stackrel{\eqref{T-1}}= \frac12 \left\langle \Lphi(w-\phi), w-\phi \right\rangle_{H^1_0(\Omega)} + R(w-\phi)\nonumber\\
   &\stackrel{\eqref{T-2}}=
   \frac12 \left\langle J'(w), \Lphi^{-1} \left( J'(w)\right) \right\rangle_{H^1_0(\Omega)} + R(w-\phi)\nonumber\\
   &\leq \frac{Q_\phi}2 \|J'(w)\|_{H^{-1}(\Omega)}^2 + R(w-\phi)
   \quad \mbox{ for } \ w \in H^1_0(\Omega),\label{GI+E}
\end{align}
where $Q_\phi$ is a positive constant given by
$$
Q_\phi := \sup \left\{ \left\langle h, \Lphi^{-1}(h) \right\rangle_{H^1_0(\Omega)} \colon h \in H^{-1}(\Omega), \ \|h\|_{H^{-1}(\Omega)} = 1 \right\} > 0.
$$
Indeed, $\Lphi$ has positive eigenvalues. Moreover, by \eqref{a:e} and \eqref{a:E}, for any $\nu > 0$ one can take $\delta_\nu > 0$ such that
  \begin{equation}\label{eE}
  |R(h)| \leq \frac\nu2\|h\|_{H^1_0(\Omega)}^2 \quad \mbox{ and } \quad
  \|r(h)\|_{H^{-1}(\Omega)} \leq \nu \|h\|_{H^1_0(\Omega)}
  \end{equation}
  for any $h \in H^1_0(\Omega)$ satisfying $\|h\|_{H^1_0(\Omega)} < \delta_\nu$.
Now, we see that
  \begin{align*}
   \MoveEqLeft{
   \|w-\phi\|_{H^1_0(\Omega)}
   }\\
  &= \left\| \Lphi^{-1} \circ \Lphi (w-\phi) \right\|_{H^1_0(\Omega)}\nonumber\\
  &\leq \|\Lphi^{-1}\|_{\mathscr{L}(H^{-1}(\Omega),H^1_0(\Omega))} \left\| \Lphi (w-\phi)\right\|_{H^{-1}(\Omega)}\\
  &\stackrel{\eqref{T-2}}\leq 
\|\Lphi^{-1}\|_{\mathscr{L}(H^{-1}(\Omega),H^1_0(\Omega))} \left(
   \left\| J'(w) \right\|_{H^{-1}(\Omega)} + \left\| r(w-\phi) \right\|_{H^{-1}(\Omega)}
  \right),
\end{align*}
whence it follows from \eqref{eE} that, for $0 < \nu < \|\Lphi^{-1}\|_{\mathscr{L}(H^{-1}(\Omega),H^1_0(\Omega))}^{-1}$,
\begin{align}
\|w-\phi\|_{H^1_0(\Omega)}
 \leq \frac{
\|\Lphi^{-1}\|_{\mathscr{L}(H^{-1}(\Omega),H^1_0(\Omega))}
}{
1-\nu \|\Lphi^{-1}\|_{\mathscr{L}(H^{-1}(\Omega),H^1_0(\Omega))}
} \left\| J'(w) \right\|_{H^{-1}(\Omega)}\label{wphi<J'}
\end{align}
for any $w \in H^1_0(\Omega)$ satisfying $\|w-\phi\|_{H^1_0(\Omega)} < \delta_\nu$. Hence combining \eqref{GI+E}, \eqref{eE} and \eqref{wphi<J'}, we conclude that \eqref{grad-ineq} is satisfied for any $\omega > Q_\phi^{1/2}/\sqrt{2}$ and some $\delta > 0$ small enough. This completes the proof.
 \end{proof}

Since $\partial_s (\vm)(s) = -J'(v(s))$ (see \eqref{eq:1.6}--\eqref{eq:1.8}) and $J(v(s)) > J(\phi)$ for $s > 0$, we obtain
\begin{align*}
\frac{1}{q-1} C_q^{-2} \left( \|\phi\|_{L^q(\Omega)}+\vep \right)^{-(q-2)} \omega^{-2} \left[ J(v(s)) - J(\phi) \right]\quad\\
\leq - \frac{\d}{\d s} \left[ J(v(s)) - J(\phi) \right]
\end{align*}
for $s \geq s_\vep$ with some $s_\vep > 0$ large enough so that $\sup_{s \geq s_\vep}\|v(s)-\phi\|_{H^1_0(\Omega)} < \delta$ (see \eqref{ap-conv}). Thus since $J(v(s_0)) \leq J(v_0)$,  we get
\begin{align}
 0 < J(v(s)) - J(\phi) &\leq \big[ J(v(s_0)) - J(\phi) \big] \e^{- \lambda (s-s_0)}\nonumber\\
&\leq \big[ J(v_0) - J(\phi) \big] \e^{\lambda s_0} \e^{- \lambda s} \quad \mbox{ for } \ s \geq s_0,\label{exp-conv01}
\end{align}
where $\lambda>0$ is any constant satisfying
\begin{equation}\label{PO}
\lambda < \frac{2}{q-1} C_q^{-2} \|\phi\|_{L^q(\Omega)}^{-(q-2)} Q_{\phi}^{-1}
\end{equation}
and $s_0 > 0$ is a constant depending on the choice of $\lambda$. Since $J(v(s)) \leq J(v_0)$ for $s \geq 0$, setting $C_\lambda = [J(v_0) - J(\phi)] \e^{\lambda s_0}$, we obtain
$$
0 < J(v(s)) - J(\phi) \leq C_\lambda \e^{-\lambda s} \quad \mbox{ for } \ s \geq 0.
$$

\section{Quantitative estimates for the rate of convergence}\label{S:que}

In this section, we shall establish a quantitative estimate for the rate of convergence obtained in the last section. To this end, as in~\cite{BF21}, let us introduce the following weighted eigenvalue problem:
\begin{equation}\label{ep}
 -\Delta e = \mu |\phi|^{q-2} e \ \mbox{ in } \Omega, \quad e = 0 \ \mbox{ on } \partial \Omega, 
\end{equation}
whose eigenpairs $\{(\mu_j,e_j)\}_{j=1}^\infty$ are such that
\begin{itemize}
 \item $0 < \mu_1 < \mu_2 \leq \mu_3 \leq \cdots \leq \mu_k \to +\infty$ as $k \to +\infty$,
 \item the eigenfunctions $\{e_j\}_{j=1}^\infty$ forms a CONS in $H^1_0(\Omega)$; in particular, $(e_j,e_k)_{H^1_0(\Omega)} = \delta_{jk}$ for $j,k \in \N$
\end{itemize}
(see, e.g.,~\cite{B-FA}). Here we note that $|\phi| \neq 0$ a.e.~in $\Omega$ (see~\cite{GaLi87} and~\cite{HaSi89}). Moreover, $\{-\Delta e_j\}_{j=1}^\infty$ forms a CONS in $H^{-1}(\Omega)$. In particular, if $\phi$ is a positive solution to \eqref{eq:1.10}, \eqref{eq:1.11}, then $\mu_1 = \lambda_q$ and $e_1 = \phi/\|\phi\|_{H^1_0(\Omega)}$.

For every $u \in H^1_0(\Omega)$, there exists a sequence $\{\alpha_j\}_{j=1}^\infty$ in $\ell^2$ such that
$$
u = \sum_{j=1}^\infty \alpha_j e_j \ \mbox{ in } H^1_0(\Omega).
$$
Hence
\begin{align*}
\Lphi(u) &= \sum_{j=1}^\infty \alpha_j \Lphi(e_j)\\
& = \sum_{j=1}^\infty \alpha_j \left[ - \Delta e_j - \lambda_q (q-1) |\phi|^{q-2} e_j \right]\\
&= \sum_{j=1}^\infty \alpha_j \frac{\mu_j - \lambda_q (q-1)}{\mu_j} (-\Delta e_j)  \quad \mbox{ in } H^{-1}(\Omega).
\end{align*}
In what follows, we shall write $\nu_j := \mu_j - \lambda_q (q-1)$ for $j \in \N$. We particularly find that
$$
\Lphi(e_j) = \nu_j |\phi|^{q-2} e_j, \quad j \in \N.
$$
For any $f \in H^{-1}(\Omega)$, since $\Linv f$ lies on $H^1_0(\Omega)$, there exists a sequence $\{\beta_j\}_{j=1}^\infty$ in $\ell^2$ such that
$$
\Linv f = \sum_{j=1}^\infty \beta_j e_j \ \mbox{ in } H^1_0(\Omega),\ \mbox{ i.e., } f = \sum_{j=1}^\infty \beta_j (-\Delta e_j)  \ \mbox{ in } H^{-1}(\Omega),
$$
and hence,
\begin{equation}\label{Linvf}
\Lphi^{-1}(f) = \sum_{j=1}^\infty \beta_j \frac{\mu_j}{\nu_j} e_j \ \mbox{ in } H^1_0(\Omega).
\end{equation}
Therefore it follows that
\begin{align*}
\left\langle f, \Lphi^{-1}(f) \right\rangle_{H^1_0(\Omega)} = \sum_{j=1}^\infty \beta_j^2 \frac{\mu_j}{\nu_j}.
\end{align*}
Noting that
$$
\|f\|_{H^{-1}(\Omega)}^2 = \sum_{j=1}^\infty \beta_j^2,
$$
we observe that
\begin{align*}
Q_\phi &= \sup \left\{ \left\langle f, \Lphi^{-1}(f) \right\rangle_{H^1_0(\Omega)} \colon f \in H^{-1}(\Omega), \ \|f\|_{H^{-1}(\Omega)} = 1 \right\}\\
&= \max_{j} \frac{\mu_j}{\nu_j} = \frac{\mu_k}{\nu_k} > 0,
\end{align*}
where $k \in \N$ is the number determining \eqref{lambda} (i.e., $\mu_k$ is the least eigenvalue for \eqref{wep} greater than $\lambda_q(q-1)$).

Thus combining the observation above with \eqref{PO}, we conclude that
$$
0 < \lambda < \frac{2}{q-1} C_q^{-2} \|\phi\|_{L^q(\Omega)}^{-(q-2)} \frac{\mu_k - \lambda_q(q-1)}{\mu_k}.
$$
Consequently, we obtain
\begin{lemma}[Exponential convergence of energy]\label{P:sc-conv}
Let $v = v(x,s)$ be a {\rm(}possibly sign-changing{\rm)} weak solution to \eqref{eq:1.6}--\eqref{eq:1.8} and let $\phi = \phi(x)$ be a {\rm(}possibly sign-changing{\rm)} nontrivial solution to \eqref{eq:1.10}, \eqref{eq:1.11} such that $v(s_n) \to \phi$ strongly in $H^1_0(\Omega)$ for some $s_n \to +\infty$. Suppose that $\phi$ is non-degenerate. Then for any constant $\lambda > 0$ satisfying
\begin{equation}\label{exp-est0}
0 < \lambda < \frac 2{q-1} C_q^{-2} \|\phi\|_{L^q(\Omega)}^{-(q-2)} \frac{\mu_k - \lambda_q(q-1)}{\mu_k},
\end{equation}
where $\mu_k$ is the least eigenvalue for \eqref{wep} greater than $\lambda_q(q-1)$ and $C_q$ is the best constant of the Sobolev-Poincar\'e inequality \eqref{SI}, there exists a constant $C_\lambda > 0$ depending on the choice of $\lambda$ such that
$$
0 \leq J(v(s)) - J(\phi) \leq C_\lambda \e^{-\lambda s} \quad \mbox{ for } \ s \geq 0.
$$
\end{lemma}

\begin{remark}[Least-energy asymptotic profiles]\label{R:leap}
In particular, if $\phi>0$ is a \emph{least-energy} solution to \eqref{eq:1.10}, \eqref{eq:1.11}, it then holds that
$$
C_q = \frac{\|\phi\|_{L^q(\Omega)}}{\|\nabla \phi\|_{L^2(\Omega)}} = \lambda_q^{-1/2} \|\phi\|_{L^q(\Omega)}^{(2-q)/2},
$$
(see~\cite{Rabinowitz,Willem} and also~\cite{BGV12,BGV13} for $q$ close to $2$) and hence, we can choose any $\lambda$ satisfying
\begin{equation*}
0 < \lambda < \frac{2\lambda_q}{q-1} \frac{\mu_k - \lambda_q(q-1)}{\mu_k} = \lamo \frac{\lambda_q}{\mu_k} = \lamo \frac{\mu_1}{\mu_k}.
\end{equation*}
Here we used the fact that $\mu_1 = \lambda_q$ because of $\phi > 0$. Moreover, noting that $\mu_1 < \mu_k$, we note that in Theorem \ref{T:sc-conv} there still remains a gap from the sharp rate $\lamo$ even for least-energy asymptotic profiles (cf.~Corollary \ref{C:H10}).
\end{remark}

\section{Exponential convergence of rescaled solutions}\label{S:exp-conv}

In this section, we shall derive exponential convergence of rescaled solutions $v = v(x,s)$ in $H^1_0(\Omega)$ as $s \to +\infty$. From \eqref{2} along with \eqref{grad-ineq}, we observe that
\begin{align*}
\MoveEqLeft{
\omega^{-1} \left[ J(v(s)) - J(\phi) \right]^{1/2} \|\partial_s (\vm)(s)\|_{H^{-1}(\Omega)} 
}\\
&\leq \|\partial_s (\vm)(s)\|_{H^{-1}(\Omega)}^2
\leq - C \frac{\d}{\d s} \left[ J(v(s))-J(\phi)\right],
\end{align*}
whence it follows that
$$
\|\partial_s (\vm)(s)\|_{H^{-1}(\Omega)} \leq - C \frac{\d}{\d s} \left[ J(v(s))-J(\phi)\right]^{1/2}.
$$
Thus one can derive that 
\begin{align*}
 \MoveEqLeft{
 \left\| |\phi|^{q-2}\phi - (\vm)(s) \right\|_{H^{-1}(\Omega)}
 }\\
 &\leq
 \int^\infty_s \left\| \partial_s \left( \vm \right) (\sigma) \right\|_{H^{-1}(\Omega)} \, \d \sigma\\
 &\leq C
 \left[ J(v(s)) - J(\phi) \right]^{1/2} 
 \leq M \e^{- \frac{\lambda}2 s} \quad \mbox{ for } \ s \geq 0
\end{align*}
for some constant $M > 0$. Here we have used Lemma \ref{P:sc-conv} with some $\lambda > 0$ satisfying \eqref{exp-est0}. Then one has
\begin{align}
\MoveEqLeft{
\frac{4}{qq'} \left\| (|v|^{(q-2)/2}v)(s) - |\phi|^{(q-2)/2}\phi \right\|_{L^2(\Omega)}^2 
}\nonumber\\
 &\leq \left\langle (\vm)(s)-|\phi|^{q-2}\phi, v(s)-\phi \right\rangle_{H^1_0(\Omega)}\nonumber\\
 &\leq \left\| (\vm)(s)-|\phi|^{q-2}\phi \right\|_{H^{-1}(\Omega)} \|v(s)-\phi\|_{H^1_0(\Omega)}\nonumber\\
 &\leq CM \e^{- \frac{\lambda}2 s} \quad \mbox{ for } \ s \geq 0\label{chk1}.
\end{align}
Here we used the inequality 
\begin{equation}\label{fundame}
\frac 4{qq'} \left||a|^{(q-2)/2}a - |b|^{(q-2)/2}b\right|^2 \leq \left(|a|^{q-2}a - |b|^{q-2}b\right)(a-b)
\end{equation}
for $a,b \in \R$ (see Appendix \ref{S:fe}) as well as the fact that $\sup_{s \geq 0}\|v(s)\|_{H^1_0(\Omega)} < +\infty$. 
Moreover, using Taylor's theorem (see Theorem \ref{T:Taylor} and Remark \ref{R:Taylor} in Appendix), we observe that
\begin{align}
\MoveEqLeft{
J(v(s))- J(\phi)
}\nonumber\\
&= \frac 12 \|\nabla (v(s)-\phi)\|_{L^2(\Omega)}^2 + \Big(  \nabla \phi , \nabla (v(s)-\phi) \Big)_{L^2(\Omega)}\nonumber\\
&\quad - \frac{\lambda_q}q \|v(s)\|_{L^q(\Omega)}^q + \frac{\lambda_q}q \|\phi\|_{L^q(\Omega)}^q\nonumber\\
&= \frac 12 \|\nabla (v(s)-\phi)\|_{L^2(\Omega)}^2 + \lambda_q \int_\Omega |\phi|^{q-2}\phi(v(s)-\phi) \, \d x\nonumber\\
&\quad - \frac{\lambda_q}q \|v(s)\|_{L^q(\Omega)}^q + \frac{\lambda_q}q \|\phi\|_{L^q(\Omega)}^q\nonumber\\
&=  \frac 12 \|\nabla (v(s)-\phi)\|_{L^2(\Omega)}^2 - \dfrac{\lambda_q}2 (q-1) \int_\Omega |v(s)-\phi|^2 |\phi|^{q-2} \, \d x\nonumber\\
&\quad + o \left( \|v(s)-\phi\|_{H^1_0(\Omega)}^2 \right).\label{J->H10}
\end{align}
One can verify that
\begin{align}
\MoveEqLeft{
|v(x,t)-\phi(x)|^2
}\nonumber\\
&\leq C |\phi(x)|^{2-q} \left| |v(x,t)|^{(q-2)/2}v(x,t) - |\phi(x)|^{(q-2)/2}\phi(x) \right|^2,\label{BF1}
\end{align}
whenever $\phi(x) \neq 0$. Here we used the inequality, 
\begin{equation}\label{sublin-ineq2}
0 \leq \frac{|a|^{p-1}a - |b|^{p-1}b}{a-b} \leq 2^{1-p} |a|^{p-1}  \quad \mbox{ for } \ a,b \in \R, \ a \neq 0
\end{equation}
for $p \in (0,1)$ (see Appendix \ref{S:fe}), with the choice $p = 2/q \in (0,1)$, $a = |\phi|^{(q-2)/2}\phi$ and $b = |v|^{(q-2)/2}v$. Therefore it follows from \eqref{J->H10} and \eqref{BF1} that
\begin{align*}
\MoveEqLeft{
J(v(s))-J(\phi)
}\\
&\geq \frac12 \|\nabla v(s)-\nabla \phi\|_{L^2(\Omega)}^2 - C \left\| (|v|^{(q-2)/2}v)(s) - |\phi|^{(q-2)/2}\phi \right\|_{L^2(\Omega)}^2 \\
&\quad + o \left( \|v(s)-\phi\|_{H^1_0(\Omega)}^2 \right). 
\end{align*}
Combining all these facts (see Lemma \ref{P:sc-conv} and \eqref{chk1}), we deduce that
$$
\|\nabla v(s)-\nabla \phi\|_{L^2(\Omega)}^2 \leq C \left( \e^{-\lambda s} + \e^{-\frac\lambda2s}\right) \lesssim \e^{-\frac\lambda2s}
$$
for $s \gg 1$. Now, turning back to \eqref{chk1} with the above, we can derive that
$$
\left\| (|v|^{(q-2)/2}v)(s) - |\phi|^{(q-2)/2}\phi \right\|_{L^2(\Omega)}^2
\lesssim \e^{- \frac \lambda 2 s} \e^{-\frac \lambda 4 s}
= \e^{-\frac \lambda 2 (1 + \frac 12)s},
$$
which also leads us to obtain
$$
\|\nabla v(s)-\nabla \phi\|_{L^2(\Omega)}^2 \lesssim \e^{-\frac \lambda 2 (1 + \frac 12)s}.
$$
Iterating these procedures, we can conclude that, for any $\mu < \lambda$, there exists a constant $C_\mu$ depending on the choice of $\mu$ such that
\begin{equation}\label{v-exp-c}
\|\nabla v(s)-\nabla \phi\|_{L^2(\Omega)}^2 \leq C_\mu \e^{-\mu s} \quad \mbox{ for } \ s \geq 0.
\end{equation}
Thus we obtain

\begin{lemma}[Exponential convergence of rescaled solutions]\label{P:exp-conv}
Under the same assumptions as in Lemma \ref{P:sc-conv}, if $J(v(s))-J(\phi)$ converges to zero at an exponential rate $\e^{-\lambda s}$ as $s \to +\infty$, then, for any $0 < \mu < \lambda$, it holds that $v(s) \to \phi$ strongly in $H^1_0(\Omega)$ at the rate $\e^{-\mu s/2}$ as $s \to +\infty$.
\end{lemma}

\begin{proof}[Proof of Theorem \ref{T:sc-conv}]
Theorem \ref{T:sc-conv} can be proved by combining Lemmata \ref{P:sc-conv} and \ref{P:exp-conv}. To be more precise, first fix $\lambda$ satisfying \eqref{exp-est}, and then, take another $\lambda'$ which is greater than $\lambda$ but still satisfies \eqref{exp-est}. Then apply Lemma \ref{P:sc-conv} for the choice $\lambda'$ to get the decay of $J(v(s)) - J(\phi)$ at the rate $\e^{-\lambda' s}$. Finally, apply Lemma \ref{P:exp-conv} by substituting $\lambda$ and $\lambda'$ to $\mu$ and $\lambda$ of the lemma, respectively, to get the conclusion.
\end{proof}

\section{Almost sharp rate of convergence for positive asymptotic profiles}\label{S:almost_sharp}

In Theorem \ref{T:sc-conv}, the rate of convergence \eqref{c:J-exp:sc} is estimated by \eqref{exp-est}; however, it is still suboptimal (even for least-energy solutions, see Remark \ref{R:leap}). In Sections \ref{S:almost_sharp}--\ref{S:2<q<3}, we shall more precisely estimate the rate of convergence for \emph{non-negative} rescaled solutions to non-degenerate \emph{positive} asymptotic profiles. We assume that $u_0 \geq 0$ a.e.~in $\Omega$, and hence, $v = v(x,s)$ is always non-negative in $\Omega \times (0,+\infty)$. In what follows, we let $k \in \N$ be such that $\nu_k > 0$ and $\nu_\ell < 0$ for $\ell = 1,2,\ldots,k-1$. Moreover, we denote by $L^2(\Omega;\phi^{q-2}\d x)$ and $L^2(\Omega;\phi^{2-q}\d x)$ the spaces of square-integrable functions with weights $\phi(x)^{q-2}$ and $\phi(x)^{2-q}$, respectively. 

Moreover, we shall use the following fact:
\begin{equation}\label{del-decay}
\delta(s) := \left\|\frac{v(s)}{\phi}-1\right\|_{L^\infty(\Omega)} \to 0 \quad \mbox{ as } \ s \to +\infty.
\end{equation}
It was first proved by~\cite[Theorem 2.1]{BGV12} based on the global Harnack principle developed by~\cite[Proposition 6.2]{DKV91}, where $\Omega$ is supposed to be of class $C^2$, and then, it was extended to a quantitative convergence by~\cite{BF21} with a proof independent of~\cite{BGV12} and using only the $C^{1,1}$ regularity of $\Omega$ (see \eqref{quan-re2} and \eqref{quan-re3} in Lemma \ref{L:re-exp} below). Therefore (for bounded $C^{1,1}$ domains, using Theorem \ref{T:sc-conv} and Lemma \ref{L:re-exp} below) we can take $s_1 > 0$ large enough so that
\begin{equation}\label{ghp}
0 < \frac 12 \phi \leq v(s) \leq \frac 32 \phi \ \mbox{ a.e.~in } \Omega \quad \mbox{ for } \ s \geq s_1
\end{equation}
(cf.~see~\cite[Proposition 6.2]{DKV91}). Hence since $v(s)/\phi$ is bounded a.e.~in $\Omega$ for $s > s_1$, noting that $\partial_s (v^{q/2})(s) \in L^2(\Omega)$ by \eqref{en}, we find from \eqref{pvm-exp} that
\begin{align*}
\MoveEqLeft{
\int_\Omega \left|\partial_s (v^{q-1})(s)\right|^2 \phi^{2-q} \, \d x
}\nonumber\\
&= \frac{4(q-1)^2}{q^2} \int_\Omega \left|\partial_s (v^{q/2})(s)\right|^2 \left(\frac{v(s)}{\phi}\right)^{q-2} \, \d x < +\infty,
\end{align*}
which along with \eqref{eq:1.6} implies $J'(v(s)) \in L^2(\Omega;\phi^{2-q}\d x)$, for $s > s_1$. Therefore, due to \eqref{en} and \eqref{del-decay}, for any $\vep > 0$, one can take $s_\vep > s_1$ large enough that
\begin{align}
\|J'(v(s))\|_{L^2(\Omega;\phi^{2-q}\d x)}^2
&\leq \frac{4(q-1)^2}{q^2} (1+\vep)^{q-2} \int_\Omega \left|\partial_s (v^{q/2})(s)\right|^2 \, \d x\nonumber\\
&\leq - (q-1)(1+\vep)^{q-2} \dfrac{\d}{\d s} J(v(s))\label{J-EI}
\end{align}
for $s \geq s_\vep$.

With the aid of Taylor's theorem in Banach spaces, we can obtain the following:

\begin{lemma}\label{L:Taylor}
For each $s > s_1$, it holds that 
\begin{align}
J(v(s)) - J(\phi) &= \frac 1 2 \left\langle \Lphi(v(s)-\phi), v(s)-\phi \right\rangle_{H^1_0(\Omega)} + E(s),\label{J-exp}\\
J'(v(s)) &= \Lphi(v(s)-\phi) + e(s).\label{J'-exp}
\end{align}
Here and henceforth, $E(s) \in \R$ and $e(s) \in H^{-1}(\Omega)$ denote generic functions satisfying
\begin{equation}\label{E-weak}
\lim_{s \to \infty} \frac{|E(s)|}{\|v(s)-\phi\|_{H^1_0(\Omega)}^{2+\gamma}} < +\infty,\quad
 \lim_{s \to \infty} \frac{\|e(s)\|_{H^{-1}(\Omega)}}{\|v(s)-\phi\|_{H^1_0(\Omega)}^{1+\gamma}} < +\infty
\end{equation}
for some $\gamma \in (0,1]$ and may vary from line to line.
\end{lemma}

\begin{proof}
In case $q \geq 3$, $J$ is of class $C^3$ in $H^1_0(\Omega)$ in the sense of Fr\'echet derivative (this fact may be standard, but it will be checked in Appendix \ref{S:Taylor}). Hence employing Taylor's theorem (see Theorem \ref{T:Taylor} in Appendix \ref{S:Taylor}) and recalling that $J'(\phi) = 0$, we can immediately verify \eqref{J-exp} and \eqref{J'-exp} with $E(s) \in \R$ and $e(s) \in H^{-1}(\Omega)$ satisfying \eqref{E-weak} with $\gamma = 1$. In case $2 < q < 3$, $J''$ may fail to be Fr\'echet differentiable at $\phi$ in $H^1_0(\Omega)$; however, we can still prove the assertions for some $\gamma \in (0,1)$. A proof for this case will be detailed in Section \ref{S:2<q<3}.
\end{proof}

Let $s > 0$ be fixed for a while. Since $\Lphi$ is invertible, one can deduce from \eqref{J-exp} and \eqref{J'-exp} along with \eqref{E-weak} that
\begin{align}\label{J-expand}
J(v(s)) - J(\phi) = \frac 1 2 \left\langle J'(v(s)), \Lphi^{-1}(J'(v(s))) \right\rangle_{H^1_0(\Omega)} + E(s).
\end{align}
Since $J'(v(s))$ belongs to $H^{-1}(\Omega)$, there exists a sequence $\{\sigma_j(s)\}_{j=1}^\infty$ in $\ell^2$ such that
$$
J'(v(s)) = \sum_{j=1}^\infty \sigma_j(s) (-\Delta e_j) \ \mbox{ in } H^{-1}(\Omega)
$$
(namely, we set $\sigma_j(s) = (J'(v(s)),-\Delta e_j)_{H^{-1}(\Omega)}$ for $j \in \N$). Hence, by virtue of \eqref{Linvf},
\begin{align*}
\Lphi^{-1}(J'(v(s))) = \sum_{j=1}^\infty \sigma_j(s) \frac{\mu_j}{\nu_j} e_j  \ \mbox{ in } H^1_0(\Omega).
\end{align*}
Thus
\begin{align*}
\MoveEqLeft{
\frac 1 2 \left\langle J'(v(s)), \Lphi^{-1}(J'(v(s))) \right\rangle_{H^1_0(\Omega)}
}\\
&= \frac 1 2 \sum_{i=1}^\infty \sum_{j=1}^\infty \sigma_i(s) \sigma_j(s)\frac{\mu_j}{\nu_j} \langle -\Delta e_i , e_j \rangle_{H^1_0(\Omega)}
= \frac 1 2 \sum_{j=1}^\infty \sigma_j(s)^2 \frac{\mu_j}{\nu_j}.
\end{align*}
Consequently,
\begin{align*}
J(v(s)) - J(\phi) - \frac 1 2 \sum_{j=1}^{k-1} \sigma_j(s)^2 \frac{\mu_j}{\nu_j}
&= \frac 1 2 \sum_{j=k}^\infty \sigma_j(s)^2 \frac{\mu_j}{\nu_j} + E(s)\\
&\leq \frac 1 {2\nu_k} \sum_{j=k}^\infty \mu_j \sigma_j(s)^2 + E(s).
\end{align*}
On the other hand, we can check in a standard manner that $\{-\Delta e_j/\sqrt{\mu_j}\}_{j=1}^\infty$ forms a CONS in $L^2(\Omega;\phi^{2-q}\d x)$ equipped with the inner product 
$$
(f,g)_{L^2(\Omega;\phi^{2-q}\d x)} = \int_\Omega f g \phi^{2-q} \, \d x \quad \mbox{ for } \ f,g \in L^2(\Omega;\phi^{2-q}\d x).
$$ 
Moreover, since $J'(v(s))$ belongs to $L^2(\Omega;\phi^{2-q}\d x)$, recalling \eqref{wep} and noting that $\langle f,u \rangle_{H^1_0(\Omega)} = (f, -\Delta u)_{H^{-1}(\Omega)}$ for $u \in H^1_0(\Omega)$ and $f \in H^{-1}(\Omega)$ (see \eqref{H-1product}), we observe that
\begin{align*}
\MoveEqLeft{
 \left( J'(v(s)), \frac{-\Delta e_j}{\sqrt{\mu_j}} \right)_{L^2(\Omega; \phi^{2-q} \d x)}
}\\
&= \int_\Omega J'(v(s)) \frac{-\Delta e_j}{\sqrt{\mu_j}} \phi^{2-q} \, \d x \\
&\stackrel{\eqref{wep}}= \sqrt{\mu_j} \int_\Omega J'(v(s)) e_j \, \d x\\
&= \sqrt{\mu_j} \langle J'(v(s)), e_j \rangle_{H^1_0(\Omega)}\\
&\stackrel{\eqref{H-1product}}= \sqrt{\mu_j} \left( J'(v(s)), -\Delta e_j \right)_{H^{-1}(\Omega)}
= \sqrt{\mu_j} \sigma_j(s)
\end{align*}
for $j \in \N$. Therefore we have
$$
J'(v(s)) = \sum_{j=1}^\infty \sqrt{\mu_j} \sigma_j(s) \frac{-\Delta e_j}{\sqrt{\mu_j}} \ \mbox{ in } L^2(\Omega;\phi^{2-q}\d x),
$$
which implies
\begin{align*}
 \sum_{j=k}^\infty \mu_j \sigma_j(s)^2
 \leq \sum_{j=1}^\infty \mu_j \sigma_j(s)^2
 = \|J'(v(s))\|_{L^2(\Omega;\phi^{2-q}\d x)}^2.
\end{align*}
Thus we obtain
$$
J(v(s)) - J(\phi) - \frac 1 2 \sum_{j=1}^{k-1} \sigma_j(s)^2 \frac{\mu_j}{\nu_j}
\leq \frac 1 {2\nu_k} \|J'(v(s))\|_{L^2(\Omega;\phi^{2-q}\d x)}^2 + E(s).
$$
Moreover, since $J'(v(s)) = \Lphi (v(s)-\phi) + e(s)$ and $\Lphi$ is invertible, we observe that
\begin{align*}
\MoveEqLeft{
\|v(s)-\phi\|_{H^1_0(\Omega)} 
}\\
&= \|\Lphi^{-1} (J'(v(s)) - e(s)) \|_{H^1_0(\Omega)} \\
&\leq \|\Lphi^{-1}\|_{\mathcal{L}(H^{-1}(\Omega);H^1_0(\Omega))} \left( \|J'(v(s))\|_{H^{-1}(\Omega)} + \|e(s)\|_{H^{-1}(\Omega)} \right)\\
&\stackrel{\eqref{E-weak}}\leq \|\Lphi^{-1}\|_{\mathcal{L}(H^{-1}(\Omega);H^1_0(\Omega))} \|J'(v(s))\|_{H^{-1}(\Omega)} + \frac 12 \|v(s)-\phi\|_{H^1_0(\Omega)} 
\end{align*}
for $s$ large enough (i.e., $\|v(s)-\phi\|_{H^1_0(\Omega)} \ll 1$ by \eqref{ap-conv}). Hence we find that
\begin{align*}
E(s) &\leq C \|v(s)-\phi\|_{H^1_0(\Omega)}^{2+\gamma}\\
&\leq C \|v(s)-\phi\|_{H^1_0(\Omega)}^\gamma \|\Lphi^{-1}\|_{\mathcal{L}(H^{-1}(\Omega);H^1_0(\Omega))}^2 \|J'(v(s))\|_{H^{-1}(\Omega)}^2\\
&\leq C \|v(s)-\phi\|_{H^1_0(\Omega)}^\gamma \|\Lphi^{-1}\|_{\mathcal{L}(H^{-1}(\Omega);H^1_0(\Omega))}^2 \|J'(v(s))\|_{L^2(\Omega;\phi^{2-q}\d x)}^2\\
&=: \beta(s) \|J'(v(s))\|_{L^2(\Omega;\phi^{2-q}\d x)}^2
\end{align*}
for $s$ large enough. Hence,
\begin{align}
\MoveEqLeft{
J(v(s)) - J(\phi) - \frac 1 2 \sum_{j=1}^{k-1} \sigma_j(s)^2 \frac{\mu_j}{\nu_j}
}\nonumber\\
&\leq \left( \frac 1 {2\nu_k} + \beta(s) \right) \|J'(v(s))\|_{L^2(\Omega;\phi^{2-q}\d x)}^2.\label{J-GI}
\end{align}
We also note that $\beta(s) \to 0$ as $s \to +\infty$, and in particular, we have $\beta(s) < \vep$ for $s \geq s_\vep$ large enough. Thus it follows from \eqref{J-EI} that
\begin{align*}
\MoveEqLeft{
J(v(s)) - J(\phi) - \frac 1 2 \sum_{j=1}^{k-1} \sigma_j(s)^2 \frac{\mu_j}{\nu_j}
}\\
&\leq - \left( \frac{1}{2\nu_k} + \vep \right) (q-1)(1+\vep)^{q-2} \frac{\d}{\d s} J(v(s)) \quad \mbox{ for } \ s \geq s_\vep,
\end{align*}
whence it follows that, for any $0 < \lambda < 2\nu_k/(q-1)$, one can take $s_1 > 0$ such that
\begin{equation*}
J(v(s)) - J(\phi) \leq - \frac 1 {\lambda} \frac{\d}{\d s} \left[ J(v(s)) - J(\phi) \right] \quad \mbox{ for } \ s \geq s_1.
\end{equation*}
Eventually, we conclude that
\begin{align}
0 < J(v(s))-J(\phi) &\leq \left[J(v(s_1))-J(\phi)\right] \e^{-\lambda (s-s_1)}\nonumber\\
&\leq \left[J(v_0)-J(\phi)\right] \e^{\lambda s_1} \e^{-\lambda s}\label{almost_sharp}
\end{align}
for all $s \geq s_1$. It is noteworthy that the exponent
$$
 \frac{2\nu_k}{q-1} = \frac{2}{q-1} \left[ \mu_k - \lambda_q(q-1)\right] = \lamo > 0
$$
is the sharp rate of convergence for solutions to the linearized problem (see \S 1 and~\cite[\S 2]{BF21} with Remark \ref{R:BF21}).

\begin{remark}[Almost sharp rate]\label{R:al_sh}
{\rm In order to verify \eqref{almost_sharp}, we do not need the differentiability of $J''$ at $\phi$ in $H^1_0(\Omega)$. Indeed, the argument so far runs as well even for $E(s) = o(\|v(s)-\phi\|_{H^1_0(\Omega)}^2)$ and $e(s) = o(\|v(s)-\phi\|_{H^1_0(\Omega)})$ as $s \to +\infty$. On the other hand, \eqref{E-weak} will be needed for deriving the sharp rate of convergence (see next section).}
\end{remark}

\section{Convergence with the sharp rate}\label{S:sharp}

Now, let us move on to a proof for the convergence with the sharp rate $\lamo$. We first recall that
\begin{align*}
0 &< J(v(s)) - J(\phi) 
\\
&\leq - \left( \frac{q-1}{2\nu_k} + (q-1)\beta(s) \right) (1+\delta(s))^{q-2} \frac{\d}{\d s} J(v(s))
\end{align*}
and $\beta(s) \leq C \|v(s)-\phi\|_{H^1_0(\Omega)}^\gamma$ for some $\gamma \in (0,1]$. Then we have
\begin{align*}
\left[ \left( \frac{q-1}{2\nu_k} + (q-1)\beta(s) \right) (1+\delta(s))^{q-2}\right]^{-1} \left[ J(v(s)) - J(\phi) \right]\quad
\\
\leq - \frac{\d}{\d s} \left[ J(v(s)) - J(\phi) \right].
\end{align*}
Furthermore, using Theorem 4.1 of~\cite{BF21} on a weighted smoothing effect that allows us to bound quantitatively the uniform relative error in terms of the weighted $L^2$ norm, we can derive an exponential convergence of the relative error from Theorem \ref{T:sc-conv}. More precisely, we have
\begin{lemma}\label{L:re-exp}
If $\|v(s)-\phi\|_{H^1_0(\Omega)} \lesssim \e^{-\mu s}$ for some constant $\mu > 0$ and any $s > 0$ large enough, then there exist constants $C, b, s_* > 0$ such that
$$
\delta(s) = \left\| \frac{v(s)}\phi - 1\right\|_{L^\infty(\Omega)} \leq C \e^{-bs}
$$
for all $s \geq s_*$.
\end{lemma}

\begin{proof}
Since $\Omega$ is a bounded $C^{1,1}$ domain, as in Theorem 4.1 of~\cite{BF21}, we can verify that there exist positive constants $C, L, s_*$ such that
\begin{align}
\MoveEqLeft{
\left\|\frac{v(s)}{\phi}-1\right\|_{L^\infty(\Omega)}
}\nonumber\\
 &\leq C \frac{\e^{L (s-s_0)}}{s-s_0} (1 + s - s_0) \sup_{\sigma \in [s_0,+\infty)} \left( \int_\Omega |v(\sigma)^{q-1}-\phi^{q-1}| \, \d x \right)^{\frac 1N}\nonumber\\
 &\quad + C (s-s_0) \e^{L(s-s_0)}\label{quan-re}
\end{align}
for any $s > s_0 \geq s_*$. Let $s > 0$ be large enough and set $s_0 = s - \e^{-a s}$, where $a$ is a positive number to be determined later. Then
\begin{align}
\MoveEqLeft{
 \left\|\frac{v(s)}{\phi}-1\right\|_{L^\infty(\Omega)}
}\nonumber\\
 &\leq C \frac{\e^{L \e^{-a s}}}{\e^{-a s}} (1 + \e^{-as}) \sup_{\sigma \in [s-\e^{-a s},+\infty)} \left( \int_\Omega |v(\sigma)^{q-1}-\phi^{q-1}| \, \d x \right)^{\frac 1N}\nonumber\\
&\quad + C \e^{-a s} \e^{L \e^{-a s}}.\label{quan-re2}
\end{align}
Moreover, we observe that
\begin{align}
\int_\Omega |v(\sigma)^{q-1}-\phi^{q-1}| \, \d x
&\leq C \int_\Omega \left| v(\sigma)-\phi \right| \, \d x\nonumber\\
&\leq C \|v(\sigma)-\phi\|_{H^1_0(\Omega)},\label{quan-re3}
\end{align}
where the constant $C$ above depend on $\|\phi\|_{L^\infty(\Omega)}$ and $\sup_{\sigma \geq s_*} \|v(\sigma)\|_{L^\infty(\Omega)}$ (see~\cite[Lemma 1]{A16}). 
Thus the assumption yields
$$
\delta(s) = \left\|\frac{v(s)}{\phi}-1\right\|_{L^\infty(\Omega)}
 \leq C \e^L \e^{as} (1 + \e^{-as}) \e^{-\frac{\mu}{N} (s-1)} + C \e^{-a s} \e^L.
$$ 
Hence it suffices to choose $0 < a < \mu/N$.
\end{proof}

Here we remark that the assumption of the lemma above can be verified with the aid of \eqref{almost_sharp} along with Lemma \ref{P:exp-conv} (or Theorem \ref{T:sc-conv} directly). Hence it follows that
$$
\beta(s) + \delta(s) \leq C \e^{-cs} \quad \mbox{ for all } \ s \geq s_*
$$
for some $c,C,s_* > 0 $. Therefore we observe that
\begin{align*}
\MoveEqLeft{
\left( \frac{q-1}{2\nu_k} + (q-1)\beta(s) \right) (1+\delta(s))^{q-2}
}\\
&\leq \frac{q-1}{2\nu_k} \left(1 + C \e^{-ds}\right) \quad \mbox{ for all } \ s \geq s_*
\end{align*}
for some $d,C > 0$. Hence
\begin{align*}
\MoveEqLeft{
\left[ \left( \frac{q-1}{2\nu_k} + (q-1)\beta(s) \right) (1+\delta(s))^{q-2} \right]^{-1}
}\\
&\geq \frac{2\nu_k}{q-1} \left(1 - \frac{C \e^{-ds}}{1+C \e^{-ds}} \right)
\geq \frac{2\nu_k}{q-1} \left(1 - C \e^{-ds}\right)
\end{align*}
for $s \geq s_*$. Thus $H(s) := J(v(s))-J(\phi) > 0$ satisfies
\begin{align*}
\frac{2\nu_k}{q-1} H(s)
&\leq - \frac{\d}{\d s} H(s) + C \e^{-d s} H(s)
\end{align*}
for $s \geq s_*$. Solving the differential inequality above, one deduces that
$$
H(s) \leq H(s_*) \e^{C/d} \exp\left( - \frac{2\nu_k}{q-1} (s-s_*)\right)
$$
for $s \geq s_*$. Thus we have proved the assertion of Theorem \ref{T:main} for $q \geq 3$. It remains only to prove Lemma \ref{L:Taylor} for the case that $2 < q < 3$, and it will be performed in the next section.

\section{The case where $2 < q < 3$}\label{S:2<q<3}

In this section, we shall prove Lemma \ref{L:Taylor} for $2 < q < 3$ to complete the proof of Theorem \ref{T:main}. It is standard that $J$ is of class $C^2$ in $H^1_0(\Omega)$ in the sense of Fr\'echet derivative and $J''(w) = - \Delta - \lambda_q (q-1)|w|^{q-2}$ for $w \in H^1_0(\Omega)$ (see, e.g.,~\cite[Corollary 1.13]{Willem}). On the other hand, $J'' :H^1_0(\Omega) \to \mathscr{L}(H^1_0(\Omega),H^{-1}(\Omega))$ may not be even G\^ateaux differentiable at $\phi$ anymore; however, it can be so in a stronger topology. We shall first claim that $J''$ is G\^ateaux differentiable at $\phi_\theta := \phi + \theta(v(s)-\phi) = (1-\theta) \phi + \theta v(s) > 0$ a.e.~in $\Omega$ for any $\theta \in [0,1]$ and $s > s_1 $ (see \eqref{ghp})  in the strong topology of
$$
X_1 := \left\{w \in H^1_0(\Omega) \colon w \phi^{\frac{q-3}2} \in L^{2\cdot2_*}(\Omega) \right\},
$$ 
where $2_* := (2^*)' = 2N/(N+2)$, equipped with the norm
$$
\|w\|_{X_1}^2 := \|w\|_{H^1_0(\Omega)}^2 + \|w \phi^{\frac{q-3}2}\|_{L^{2\cdot 2_*}(\Omega)}^2 \quad \mbox{ for } \ w \in X_1.
$$
Then $X_1 \hookrightarrow H^1_0(\Omega)$. Hence (the restriction) $J' : X_1 \to H^{-1}(\Omega)$ (onto $X_1$) turns out to be of class $C^1$ in $X_1$ in the sense of Fr\'echet derivative, and moreover, its derivative (still denoted by $J''$) can be regarded as a continuous map from $X_1$ into $\mathscr{L}(X_1,H^{-1}(\Omega))$. Let $u,e \in X_1$ and $t \neq 0$. Since $\phi_\theta = (1-\theta)\phi + \theta v(s) > 0$ a.e.~in $\Omega$ for $s > s_1 $, it then follows that
\begin{align*}
\MoveEqLeft{
\left| \frac{ [J''(\phi_\theta+te)](u) - [J''(\phi_\theta)](u)}t + \lambda_q (q-1) (q-2) \phi_\theta^{q-3} eu \right|
}\\
&= \lambda_q (q-1) \left| \frac{|\phi_\theta+te|^{q-2} - \phi_\theta^{q-2}}t - (q-2) \phi_\theta^{q-3}e \right| |u| \to 0 
\end{align*}
a.e.~in $\Omega$ as $t \to 0$. Moreover,
\begin{align}
 \left| \frac{|\phi_\theta+te|^{q-2} - \phi_\theta^{q-2}}t - (q-2) \phi_\theta^{q-3}e \right| |u|
\leq (q-1)\phi_\theta^{q-3} |e| |u|.
\label{J3-domin}
\end{align}
Here we used the fact that $0 < q-2 < 1$ and the inequality
\begin{equation}\label{fun_ineq}
|a^p - b^p| \leq a^{p-1}|a-b| \quad \mbox{ for any } \ a,b > 0 \ \mbox{ and } \ p \in (0,1). 
\end{equation}
Then the right-hand side of \eqref{J3-domin} belongs to $L^{2_*}(\Omega) \simeq (L^{2^*}(\Omega))^* \hookrightarrow H^{-1}(\Omega)$ due to the following fact:
\begin{align*}
|\phi_\theta^{\frac{q-3}2} u| &= \left| (1-\theta) \phi + \theta v(s) \right|^{\frac{q-3}2} |u|\\
&= \left| 1 - \theta + \theta (v(s)/\phi) \right|^{\frac{q-3}2} \phi^{\frac{q-3}2} |u|
\leq C \phi^{\frac{q-3}2} |u| \in L^{2\cdot 2_*}(\Omega).
\end{align*}
Indeed, $v(s)/\phi \geq 1/2$ a.e.~in $\Omega$ for $s > s_1$ (see \eqref{ghp}). Using Lebesgue's dominated convergence theorem, we can then deduce that $J'' : X_1 \to \mathscr{L}(X_1,H^{-1}(\Omega))$ is G\^ateaux differentiable at $\phi_\theta$. Moreover, we observe that the G\^ateaux derivative $\D_G J''(\phi_\theta) = -\lambda_q (q-1) (q-2) \phi_\theta^{q-3}$ of $J''$ at $\phi_\theta$ is bounded in $\mathscr{L}^{(2)}(X_1,H^{-1}(\Omega))$ for $\theta \in [0,1]$. Hence employing Taylor's theorem (see Theorem \ref{T:Taylor} in Appendix) and recalling $J'(\phi)=0$ and $J''(\phi) = \Lphi$, we can still verify that
\begin{equation*}
J'(v(s)) = \Lphi(v(s)-\phi) + \epsilon_1(v(s)-\phi),
\end{equation*}
where $\epsilon_1 : X_1 \to H^{-1}(\Omega)$ is a generic function fulfilling 
$$
\lim_{\|w\|_{X_1}\to 0} \frac{\|\epsilon_1(w)\|_{H^{-1}(\Omega)}}{\|w\|_{X_1}^2} < +\infty.
$$ 

In particular, we put $w = v(s)-\phi$. Then noting that $\|w/\phi\|_{L^\infty(\Omega)} = \|(v(s)-\phi)/\phi\|_{L^\infty(\Omega)}$ is uniformly bounded for $s > s_1$  (see \eqref{ghp}), we infer that
\begin{align*}
\|w\phi^{\frac{q-3}2}\|_{L^{2\cdot 2_*}(\Omega)}^2
&= \|(w/\phi)^{3-q} |w|^{q-1}\|_{L^{2_*}(\Omega)}\\
&\leq \|w/\phi\|_{L^\infty(\Omega)}^{3-q} \||w|^{q-1}\|_{L^{2_*}(\Omega)}
\leq C \|w/\phi\|_{L^\infty(\Omega)}^{3-q} \|w\|_{H^1_0(\Omega)}^{q-1},
\end{align*}
and hence, we observe that
\begin{align*}
\|w\|_{X_1}^2 &= \|w\|_{H^1_0(\Omega)}^2 + \|w\phi^{\frac{q-3}2}\|_{L^{2\cdot 2_*}(\Omega)}^2\\
&\leq \|w\|_{H^1_0(\Omega)}^2 + C \|w/\phi\|_{L^\infty(\Omega)}^{3-q} \|w\|_{H^1_0(\Omega)}^{q-1}.
\end{align*}
Set
$$
e(s) = \epsilon_1(v(s)-\phi),
$$
whence it follows that
$$
\|e(s)\|_{H^{-1}(\Omega)} \leq C\|v(s)-\phi\|_{H^1_0(\Omega)}^{1+ (q-2)} \quad \mbox{ for } \ s \gg 1.
$$

Similarly, setting 
$$
X_2 := \left\{w \in H^1_0(\Omega) \colon w \phi^{\frac{q-3}3} \in L^3(\Omega) \right\}
$$
equipped with
$$
\|w\|_{X_2}^3 := \|w\|_{H^1_0(\Omega)}^3 + \| w \phi^{\frac{q-3}3} \|_{L^3(\Omega)}^3 \quad \mbox{ for } \ w \in X_2, 
$$
(then $X_2 \hookrightarrow H^1_0(\Omega)$) and repeating the same argument as above again, we can prove that (the restriction) $J'': X_2 \to \mathscr{L}^{(2)}(X_2,\R)$ is G\^ateaux differentiable at $\phi_\theta$ in $X_2$ for any $\theta \in [0,1]$, and moreover, the G\^ateaux derivative $\D_G J''(\phi_\theta)$ is bounded in $\mathscr{L}^{(3)}(X_2,\R)$ for $\theta \in [0,1]$. Hence it follows that
\begin{align*}
 J(v(s)) = J(\phi) + \frac 1 2 \left\langle \Lphi(v(s)-\phi), v(s)-\phi\right\rangle_{H^1_0(\Omega)} + \epsilon_2(v(s)-\phi),
\end{align*}
where $\epsilon_2 : X_2 \to \R$ is a generic function satisfying
$$
\lim_{\|w\|_{X_2} \to 0} \dfrac{|\epsilon_2(w)|}{\|w\|_{X_2}^3} < +\infty
$$
(see Theorem \ref{T:Taylor} in Appendix). Put $w = v(s)-\phi$ again. Then we find that
\begin{align*}
 \|w \phi^\frac{q-3}3\|_{L^3(\Omega)}^{3}
 &\leq \|w/\phi\|_{L^\infty(\Omega)}^{3-q} \|w\|_{L^q(\Omega)}^q\\
 &\leq C \|w/\phi\|_{L^\infty(\Omega)}^{3-q} \|w\|_{H^1_0(\Omega)}^q
\end{align*}
and that
$$
\|w\|_{X_2}^3 \leq \|w\|_{H^1_0(\Omega)}^3 + C \|w/\phi\|_{L^\infty(\Omega)}^{3-q} \|w\|_{H^1_0(\Omega)}^q.
$$
Set $E(s) = \epsilon_2(v(s)-\phi)$. Then we obtain
$$
|E(s)| \leq C\|v(s)-\phi\|_{H^1_0(\Omega)}^{2+(q-2)} \quad \mbox{ for } \ s \gg 1.
$$
Thus we have checked \eqref{J-exp} and \eqref{J'-exp} with $E(\cdot)$ and $e(\cdot)$ satisfying \eqref{E-weak} with $\gamma = q-2 \in (0,1)$, and hence, we have completed the proof of Lemma \ref{L:Taylor} for $2 < q < 3$ as well. \qed

\bigskip
Thus the proof of Theorem \ref{T:main} has been completed. We close this section with the following remark on assumptions for domains based on the arguments so far.

\begin{remark}[Assumption for domains]\label{R:C11}
All the results in \S \ref{S:Intro} can be proved for arbitrary bounded $C^{1,1}$ domains. The $C^{1,1}$ condition for domains is needed for: (i) the $C^2(\Omega) \cap C^1(\overline{\Omega})$ regularity of solutions $\phi$ to \eqref{eq:1.10}, \eqref{eq:1.11} (see, e.g.,~\cite[Theorems 9.15 and 9.19]{GT}), (ii) Hopf's lemma (see, e.g.,~\cite[\S 6.4.2]{Evans}; indeed, the interior sphere condition follows from the $C^{1,1}$ condition) and (iii) the proof for Lemma \ref{L:re-exp} in \S \ref{S:sharp}. To be more precise for (iii), in the proof of Lemma \ref{L:re-exp}, a quantitative estimate (see \eqref{quan-re}) established in Theorem 4.1 of~\cite{BF21} is employed and the estimate is proved with the use of Green function estimates under the $C^{1,1}$ condition (see~\cite{GrWi82,ChZha95}).
\end{remark}

\section{Proofs of corollaries}\label{S:Cor}

This section is devoted to proving corollaries exhibited in \S \ref{S:Intro}. We first give a proof of Corollary \ref{C:exp-stbl}.

\begin{proof}[Proof of Corollary \ref{C:exp-stbl}]
It is well known that every non-degenerate nontrivial solution to \eqref{eq:1.10}, \eqref{eq:1.11} is isolated in $H^1_0(\Omega)$ from all the other solutions (see, e.g.,~\cite[\S 5.3]{AK13}). Moreover, we recall Theorem 2 of~\cite{AK13}: Let $\varphi$ be a least-energy solution of \eqref{eq:1.10}, \eqref{eq:1.11}. If $\varphi$ is isolated in $H^1_0(\Omega)$ from all the other (sign-definite) solutions of \eqref{eq:1.10}, \eqref{eq:1.11}, then $\varphi$ is an asymptotically stable profile in the sense of Definition \ref{D:stbl}. Therefore since $\phi$ is isolated from all the other solutions to \eqref{eq:1.10}, \eqref{eq:1.11} and takes the least energy among all the nontrivial solutions of \eqref{eq:1.10}, \eqref{eq:1.11}, it turns out to be an asymptotically stable asymptotic profile in the sense of Definition \ref{D:stbl}. Hence, any (possibly sign-changing) weak solution $v = v(x,s)$ of \eqref{eq:1.6}--\eqref{eq:1.8} emanating from some small (in $H^1_0(\Omega)$) neighbourhood $B_{H^1_0(\Omega)}(\phi;\delta)$ of $\phi$ on the phase set $\mathcal{X}$ (see \eqref{phase_set}) converges to $\phi$ strongly in $H^1_0(\Omega)$ as $s \to +\infty$. Therefore Theorem \ref{T:sc-conv} can guarantee the exponential convergence. Here we note that the constant $M_\mu$ in Theorem \ref{T:sc-conv} can be chosen so as to be independent of $v_0$, whenever $\|v_0 - \phi\|_{H^1_0(\Omega)} < \delta$. Thus the exponential stability of $\phi$ has been proved.
\end{proof}

We next prove Corollary \ref{C:entropy}.

\begin{proof}[Proof of Corollary \ref{C:entropy}]
Recalling \eqref{J-EI} and \eqref{J-GI}, we see that
\begin{align*}
 \|J'(v(s))\|_{L^2(\Omega;\phi^{2-q}\d x)}
 \leq - C \dfrac{\d}{\d s} \left[ J(v(s))-J(\phi) \right]^{1/2},
\end{align*}
whence it follows from Theorem \ref{T:main} that
\begin{align*}
 \MoveEqLeft{
 \left\| \phi^{q-1} - v^{q-1}(s) \right\|_{L^2(\Omega;\phi^{2-q}\d x)}
 }\\
 &\leq
 \int^\infty_s \left\| \partial_s \left( v^{q-1} \right) (\sigma) \right\|_{L^2(\Omega;\phi^{2-q}\d x)} \, \d \sigma\\
 &\leq C \left[ J(v(s)) - J(\phi) \right]^{1/2} 
 \leq C \e^{- \frac{\lamo}2 s}.
\end{align*}
On the other hand, we observe that
\begin{align*}
\MoveEqLeft{
\int_\Omega |v(x,s) - \phi(x)|^2 \phi(x)^{q-2} \, \d x 
}\\
&\leq \int_\Omega \left|v(x,s)^{q-1} - \phi(x)^{q-1}\right|^2 \phi(x)^{2-q} \, \d x.
\end{align*}
Here we used \eqref{fun_ineq}. Thus \eqref{ent-conv} follows immediately.
\end{proof}

Let us give a proof for Corollary \ref{C:H10}.
\begin{proof}[Proof of Corollary \ref{C:H10}]
As in \eqref{J->H10} and \S \ref{S:almost_sharp} (see also Lemma \ref{L:Taylor}), we observe that
\begin{align*}
\MoveEqLeft{
J(v(s))- J(\phi)
}\\
&=  \frac 12 \|\nabla (v(s)-\phi)\|_{L^2(\Omega)}^2 - \dfrac{\lambda_q}2 (q-1) \int_\Omega |v(s)-\phi|^2 \phi^{q-2} \, \d x\\
&\quad + O \left( \|v(s)-\phi\|_{H^1_0(\Omega)}^{2+\gamma}\right)
\end{align*}
for some $\gamma \in (0,1]$. Consequently, Theorem \ref{T:main} and Corollary \ref{C:entropy} yield
$$
\|v(s)-\phi\|_{H^1_0(\Omega)}^2 \leq C \e^{-\lamo s} \quad \mbox{ for } \ s \geq 0.
$$
Finally, \eqref{J'-conv} follows immediately from \eqref{J'-exp}. This completes the proof.
\end{proof}

From the argument above, we can also observe the following:
\begin{corollary}\label{C:rate}
Under the same assumption as in Theorem \ref{T:main}, if \eqref{sharp-rate} holds for some $\lambda > 0$, then \eqref{ent-conv} and \eqref{H10-conv} hold for the same $\lambda$.
\end{corollary}

With the aid of the regularity results~\cite{JiXi01,JiXi02}, one can also improve the topology of the relative error convergence (respectively, convergence of the difference) up to $C^q(\overline{\Omega})$ (respectively, $C^{q+1}(\overline{\Omega})$) for \emph{smooth} domains (see~\cite[Corollary 1.4]{JiXi01}).

\section{Fast diffusion flows with changing signs}\label{S:sc-beh}

Although asymptotic behavior of \emph{sign-definite} solutions to the fast diffusion equation has been well studied, dynamics of \emph{sign-changing} ones has not yet been fully pursued. In particular, since sign-changing asymptotic profiles for fast diffusion are often unstable (see~\cite{AK13}), existence of (non-stationary) weak solutions of \eqref{eq:1.6}--\eqref{eq:1.8} converging to sign-changing solutions of \eqref{eq:1.10}, \eqref{eq:1.11} may be still rather nontrivial. In this section, we shall discuss such dynamics of fast diffusion flows with changing signs.

\subsection{One-dimensional case}

We first restrict ourselves to the one-dimensional case $\Omega = (0,1)$, where the set $\{\pm\phi_k \colon k \in \N\}$ of all non-trivial solutions to \eqref{eq:1.10}, \eqref{eq:1.11} consists of the unique positive solution $\phi_1 > 0$ and sign-changing ones $\phi_k$ given by
$$
\phi_k(x) = (-1)^j k^{2/(q-2)}\phi_1(kx - j), \quad x \in (j/k,(j+1)/k)
$$
for $j = 0,1,\ldots, k-1$. Hence $\pm\phi_k$ have $k-1$ zeros arranged at equal intervals in $(0,1)$ and $J(\pm\phi_1) < J(\pm\phi_2) < \cdots < J(\pm\phi_k) \to +\infty$ as $k \to +\infty$ (see~\cite[\S 5.4]{AK13} for more details). Moreover, one can verify that $\phi_k$ is non-degenerate in a standard way. Note that, for any non-negative data $u_0 \in H^1_0(0,1) \setminus \{0\}$, the solution to \eqref{eq:1.1}--\eqref{eq:1.3} with $\Omega = (0,1)$ has the positive asymptotic profile $\phi_1$ in the sense of \eqref{ap}. Furthermore, for each $k \in \N$, we can construct a solution $u = u(x,t)$ (of \eqref{eq:1.1}--\eqref{eq:1.3}) whose asymptotic profile coincides with $\phi_k$. Indeed, for instance, set $u_0(x) = \sin (k\pi x)$ for $x \in (0,1)$. Then all the zeros of $u(\cdot,t)$ do not move for $t \geq 0$. Hence the dynamics of $u(\cdot,t)$ restricted on each subinterval $(j/k,(j+1)/k)$ is reduced to those of sign-definite solutions. 

We can also construct sign-changing initial data $u_0 \in H^1_0(0,1) \setminus \{0\}$ such that the corresponding solutions of \eqref{eq:1.1}--\eqref{eq:1.3} have sign-definite asymptotic profiles and sign-changing ones having fewer zeros; hence, some zeros of such solutions move and eventually vanish. Let $u = u(x,t)$ be the solution for \eqref{eq:1.1}--\eqref{eq:1.3} in $\Omega = (0,1)$ with a smooth initial datum $u_0$ which is even with respect to $x = 1/2$, negative in $(0,a) \cup (1-a,1)$ and positive in $(a,1-a)$ for some $a \in (0,1/2)$ such that
$$
\int^1_0 (|u_0|^{q-2}u_0)(x) \, \d x > 0
$$
(hence $u_0$ has exactly two zeros in $(0,1)$). Then $u(\cdot,t)$ is also even with respect to $x=1/2$ for $t > 0$. Integrating both sides of \eqref{eq:1.1} over $\Omega = (0,1)$ and utilizing the evenness of $u(\cdot,t)$ with respect to $x = 1/2$, we observe that
$$
\frac\d{\d t} \int^1_0 (|u|^{q-2}u)(x,t) \, \d x - 2 \partial_x u(1,t) = 0.
$$
Now, suppose to the contrary that $\partial_x u(1,t) \geq 0$ for all $t \geq 0$. Then one gets
$$
\int^1_0 (|u|^{q-2}u)(x,t) \, \d x \geq \int^1_0 (|u_0|^{q-2}u_0)(x) \, \d x > 0
\quad \mbox{ for all } \ t \geq 0,
$$
which is a contradiction to the finite-time extinction of $u = u(x,t)$. Hence $\partial_x u(1,t_0) < 0$ at some $t_0 \in (0,t_*)$. Since the number of zeros of $u(\cdot,t)$ is non-increasing in $t$, $u(\cdot,t_0)$ must be non-negative in $\Omega = (0,1)$ (see, e.g.,~\cite{GalVaz04}). Therefore the solution $u = u(x,t)$ has the positive asymptotic profile $\phi_1$. Furthermore, for each $k \in \N$, extending the function $u_0$ considered above to be an anti-periodic function in $(0,k)$, i.e., $u_0(x+1) = - u_0(x)$ for $x \in (0,k-1)$, one can construct a sign-changing solution (for \eqref{eq:1.1}--\eqref{eq:1.3} with $\Omega = (0,k)$) which has a sign-changing asymptotic profile with fewer zeros (than its initial datum). 

\subsection{Multi-dimensional case}

The multi-dimensional case is more complicated; indeed, the structure of nontrivial solutions to \eqref{eq:1.10}, \eqref{eq:1.11} is not so simple as in the one-dimensional case. It is already difficult to check the non-degeneracy of sign-changing solutions (indeed, even in balls, although the positive solution is unique and non-degenerate, there exist non-radial sign-changing solutions, which are degenerate; see~\cite[Theorem 1.3]{AfPa04}). 

We shall consider dumbbell-shaped domains in $\R^N$. Set 
$$
B = B_+ \cup B_- \subset \R^N,
$$
where $B_\pm$ denotes the open unit ball in $\R^N$ centered at $x = \pm 2 e_1$, respectively, with a unit vector $e_1 \in \R^N$ (hence $\overline{B_+} \cap \overline{B_-} = \emptyset$) and let $C = \{t e_1 \colon t \in [-1,1]\}$. Moreover, let $(\Omega_n)$ be a sequence of smooth bounded domains of $\R^N$ involving $\overline{B} \cup C$ and symmetric with respect to the hyperplane 
$$
H := \{x \in \R^N \colon x \cdot e_1 = 0\}
$$
through the origin such that $\Omega_n \to B$ in a proper sense as $n \to +\infty$ (see~\cite[p.122]{Dancer1} for more details). Furthermore, let $\tilde{B} \subset \R^N$ be a ball including $\Omega_n$ for $n$ large enough.

In what follows, we let $\phi_{+-} \in H^1_0(B)$ coincide with the positive and negative radial solutions to \eqref{eq:1.10}, \eqref{eq:1.11} in $B_+$ and $B_-$, respectively (thanks to~\cite{GNN}, positive solutions in balls are radial and unique). Then $\phi_{+-}$ turns out to be a non-degenerate solution to \eqref{eq:1.10}, \eqref{eq:1.11} with $\Omega = B$ (indeed, the restriction of $\phi_{+-}$ onto each of the disjoint balls is non-degenerate due to~\cite{Pacella}). Thanks to~\cite[(i) of Theorem 1]{Dancer1}, for each $n \in \N$ large enough, there exists a \emph{non-degenerate} solution $\phi_n \in H^1_0(\Omega_n)$ of \eqref{eq:1.10}, \eqref{eq:1.11} with $\Omega = \Omega_n$ uniquely corresponding to $\phi_{+-}$ in the sense that $\phi_n \to \phi_{+-}$ strongly in $L^q(\tilde{B})$ as $n \to +\infty$ and $\phi_n$ is the only solution in $H^1_0(\Omega_n)$ close to $\phi_{+-}$ in $L^q(\tilde{B})$. Here and henceforth, we use the same notation for functions of class $H^1_0(B)$ (or $H^1_0(\Omega_n)$) and their zero extensions onto $\tilde{B}$, when no confusion can arise. Hence $\phi_n$ is \emph{sign-changing} for $n \in \N$ large enough, since so is $\phi_{+-}$. Then $(\Omega_n,\phi_n)$ will turn out to be our desired domain and asymptotic profile for fast diffusion for $n \in \N$ large enough. This fact will be precisely stated in Theorem \ref{P:odd-stbl} below. 

To this end, let us first recall several materials developed in~\cite{AK13}. The set of initial data for \eqref{eq:1.6}--\eqref{eq:1.8} via the scaling \eqref{cv} is defined as
\begin{align*}
\mathcal{X}(\Omega) :=& \left\{ t_*(u_0)^{-1/(q-2)} u_0 \colon u_0 \in H^1_0(\Omega) \setminus \{0\} \right\}\\
=& \left\{ v_0 \in H^1_0(\Omega) \colon t_*(v_0) = 1 \right\}
\end{align*}
(see~\cite[Proposition 6]{AK13} for the equality). It is noteworthy that $\mathcal{X}(\Omega)$ is homeomorphic to the unit sphere in $H^1_0(\Omega)$ (see~\cite[Proposition 10]{AK13}). We denote by $\mathcal{S}(\Omega)$ the set of all nontrivial solutions to \eqref{eq:1.10}, \eqref{eq:1.11}. We may simply write $\mathcal{X}$ and $\mathcal{S}$ instead of $\mathcal{X}(\Omega)$ and $\mathcal{S}(\Omega)$, respectively, when no confusion can arise. Then the following proposition holds true:
\begin{proposition}[Properties of the set of initial data~\cite{AK13}]\label{P:X}
It holds that\/{\rm :}
\begin{enumerate}
 \item[(i)] The set $\mathcal{S}$ is included in $\mathcal{X}$ {\rm (}see~\cite[Proposition 10]{AK13}{\rm )}.
 \item[(ii)] Moreover, the weak solution $v = v(x,s)$ emanating from $v_0 \in \mathcal{X}$ quasi-converges to a nontrivial solution for \eqref{eq:1.10}, \eqref{eq:1.11} {\rm (}see~\cite[Theorem 1]{AK13} and {\rm \S \ref{S:Intro}}{\rm )}. 
 \item[(iii)] Furthermore, $\mathcal{X}$ is an invariant set of the dynamical system generated by \eqref{eq:1.6}--\eqref{eq:1.8} {\rm (}see~\cite[Proposition 5]{AK13}{\rm )}.
 \item[(iv)] The set $\mathcal{X}$ is sequentially closed in the weak topology of $H^1_0(\Omega)$ {\rm (}see~\cite[Proposition 7]{AK13}{\rm )}.
\end{enumerate} 
\end{proposition}

Moreover, let $\mathcal{S}(B)$ be defined as above and let $\mathcal{S}_{H}(B)$ be its subset whose elements are odd with respect to the hyperplane $H$, that is, $\phi \in \mathcal{S}_H(B)$ means $\phi \in \mathcal{S}(B)$ and $\phi(x) = - \phi(\mathrm{Ref}_H(x))$ for $x \in B$, where $\mathrm{Ref}_H(x) := x - 2 (x \cdot e_1) e_1$ stands for the reflection of $x$ with respect to the hyperplane $H$. In particular, $\phi_{+-} \in \mathcal{S}_H(B)$. Moreover, set 
$$
J_B(w) := \frac 12 \int_B |\nabla w(x)|^2 \, \d x - \frac{\lambda_q}q \int_B |w(x)|^q \, \d x \quad \mbox{ for } \ w \in H^1_0(B).
$$
We define $\mathcal{S}(\Omega_n)$, $\mathcal{S}_H(\Omega_n)$ and $J_{\Omega_n}$ in an analogous way. Then we claim that 
$$
\phi_n \in \mathcal{S}_H(\Omega_n) 
$$
for $n \in \N$ large enough. Indeed, since $\phi_n \to \phi_{+-}$ strongly in $L^q(\tilde{B})$ as $n \to +\infty$, we find from the symmetry of $\Omega_n$ that $-\phi_n(\mathrm{Ref}_H(\cdot)) \to -\phi_{+-}(\mathrm{Ref}_H(\cdot)) = \phi_{+-}$ strongly in $L^q(\tilde{B})$ as $n \to +\infty$. From the uniqueness of $(\phi_n)$ (see~\cite[(i) of Theorem 1]{Dancer1}), we find that $\phi_n$ coincides with $-\phi_n(\mathrm{Ref}_H(\cdot))$, i.e., $\phi_n \in \mathcal{S}_H(\Omega_n)$, for $n \in \N$ large enough. Furthermore, we set
$$
\mathcal{X}_H(\Omega_n) = \left\{ w \in \mathcal{X}(\Omega_n) \colon w \mbox{ is odd with respect to the hyperplane } H \right\}.
$$
Then all the assertions of Proposition \ref{P:X} with $\mathcal{X}$ and $\mathcal{S}$ replaced by $\mathcal{X}_H$ and $\mathcal{S}_H$, respectively, hold true, since the oddness of initial data is inherited by the solutions to \eqref{eq:1.6}--\eqref{eq:1.8} (see~\cite[Theorem 2.5]{INdAM13}). Moreover, we stress that for any $w \in H^1_0(\Omega_n) \setminus \{0\}$ which is odd with respect to the hyperplane $H$ one can take a constant $x(w) > 0$ such that $x(w)w$ lies on the set $\mathcal{X}_H(\Omega_n)$ (more precisely, we have $x(w) = t_*(w)^{-1/(q-2)}$).

The following theorem ensures exponential stability of the asymptotic profile $\phi_n$, which is sign-changing and non-degenerate, in $\mathcal{X}_H(\Omega_n)$:
\begin{theorem}[Exponential stability of $\phi_n$ in $\mathcal{X}_H(\Omega_n)$]\label{P:odd-stbl}
Let $(\Omega_n)$ and $(\phi_n)$ be defined as above. Then, for any $n \in \N$ large enough, $\phi_n$ is exponentially stable under the dynamical system generated by \eqref{eq:1.6}--\eqref{eq:1.8} in $\mathcal{X}_H(\Omega_n)$, that is, for any $\vep > 0$ there exists $\delta_{n,\vep} > 0$ such that any weak solution $v = v(x,s)$ to \eqref{eq:1.6}--\eqref{eq:1.8} with $\Omega = \Omega_n$ satisfies
$$
\sup_{s \geq 0} \|v(s)-\phi_n\|_{H^1_0(\Omega_n)} < \vep,
$$
provided that $v(0) \in \mathcal{X}_H(\Omega_n)$ and $\|v(0)-\phi_n\|_{H^1_0(\Omega_n)} < \delta_{n,\vep}$\/{\rm ;} moreover, there exist constants $C_n, \lambda_n,\delta_{n,0} > 0$ such that any weak solution $v = v(x,s)$ to \eqref{eq:1.6}--\eqref{eq:1.8} with $\Omega = \Omega_n$ fulfills
$$
\|v(s)-\phi_n\|_{H^1_0(\Omega)} \leq C_n \e^{-\lambda_n s/2} \quad \mbox{ for all } \ s \geq 0,
$$
provided that $v(0) \in \mathcal{X}_H(\Omega_n)$ and $\|v(0)-\phi_n\|_{H^1_0(\Omega_n)} < \delta_{n,0}$. Here $\lambda_n$ can be chosen as in \eqref{exp-est} for $\phi = \phi_n$ and $\Omega = \Omega_n$.
\end{theorem}

Before proving this theorem, we recall Theorem 3 of~\cite{AK13}: Let $\psi$ be a sign-changing profile of a solution of \eqref{eq:1.1}--\eqref{eq:1.3}. If $\psi$ is isolated in $H^1_0(\Omega)$ from all the other solutions, then $\psi$ is unstable in the sense of Definition of \ref{D:stbl}. Therefore $\phi_n$ turns out to be unstable in $\mathcal{X}(\Omega_n)$, whose elements are not always odd, since $\phi_n$ is sign-changing and non-degenerate (hence isolated in $H^1_0(\Omega_n)$).

To prove Theorem \ref{P:odd-stbl}, we need the following:
\begin{lemma}\label{L:Sset0}
 There exists a constant $r_0 > 0$ such that
\begin{equation}\label{Sset0}
\left\{
\varphi \in \mathcal{S}_H(B) \colon J_B(\varphi) \leq J_B(\phi_{+-}) + r_0
\right\} = \{\pm\phi_{+-}\}.
\end{equation}
\end{lemma}

\begin{proof}
We first note that $\phi_{+-}$ attains the infimum of the energy $J_B$ over $\mathcal{S}_H(B)$, since the positive solution on each ball takes the least energy among all nontrivial solutions on the ball. We next let $\phi_{\pm\mp} \in \mathcal{S}_H(B)$ coincide with a \emph{least-energy nodal solution} $\psi \in \mathcal{S}(B_+)$ in $B_+$, that is, $\psi \in \mathcal{S}(B_+)$  is sign-changing and attains the minimum value of $J_{B_+}$ among all sign-changing solutions in $B_+$ (see~\cite{AfPa04,BaWeWi05}). Here we note that $\psi$ takes the second minimum value of $J_{B_+}$ among $\mathcal{S}(B_+)$, since the positive solution is unique in the ball $B_+$. Then from the oddness of $\phi_{\pm\mp}$ it follows that 
\begin{equation}\label{odd}
\phi_{\pm\mp}(x) = - \phi_{\pm\mp}(\mathrm{Ref}_H(x)) \quad\mbox{ for } x \in B_-.
\end{equation}
Hence $\phi_{+-}$ and $\phi_{\pm\mp}$ take the first and second minimum values of the energy $J_{B}$ among $\mathcal{S}_H(B)$, respectively. We take $0 < r_0 < J_B(\phi_{\pm\mp}) - J_B(\phi_{+-})$. Then \eqref{Sset0} follows immediately.
\end{proof}

We further need the following:

\begin{lemma}\label{l1}
Let $n \in \N$ be large enough. The functions $\phi_n$ and $- \phi_n$ are minimizers of the functional $J_{\Omega_n}$ over the set $\mathcal{X}_H(\Omega_n)$. Moreover, it holds that $J_{\Omega_n}(w) > J_{\Omega_n}(\pm\phi_n)$ for any $w \in \mathcal{X}_H(\Omega_n) \setminus \{ \pm \phi_n\}$.
\end{lemma}

\begin{proof}
 We first claim that 
\begin{equation}\label{cl1}
\left\{
\varphi \in \mathcal{S}_H(\Omega_n) \colon J_{\Omega_n}(\varphi) \leq J_B(\phi_{+-}) + r_0
\right\} = \{\pm\phi_n\}
\end{equation}
for any $n \in \N$ large enough. Here $r_0$ is given as in \eqref{Sset0}. Indeed, recalling that $\phi_n \in \mathcal{S}(\Omega_n)$, $\phi_n \to \phi_{+-}$ strongly in $L^q(\tilde{B})$ as $n \to +\infty$ and $\phi_{+-} \in \mathcal{S}(B)$, we deduce that
\begin{align*}
J_{\Omega_n}(\phi_n) &= \frac{q-2}{2q} \lambda_q \|\phi_n\|_{L^q(\Omega_n)}^q = \frac{q-2}{2q} \lambda_q \|\phi_n\|_{L^q(\tilde{B})}^q\\
&\to \frac{q-2}{2q} \lambda_q \|\phi_{+-}\|_{L^q(\tilde{B})}^q =  J_B(\phi_{+-})
\end{align*}
as $n \to +\infty$. Hence we find that the set given by the left-hand side of \eqref{cl1} includes $\pm\phi_n$ for $n \in N$ large enough. Therefore it suffices to prove the inverse inclusion. Suppose to the contrary that, up to a (not relabeled) subsequence, there exists a sequence $(\varphi_n)$ in $\mathcal{S}_H(\Omega_n) \setminus \{\pm \phi_n\}$ such that
$$
J_{\Omega_n}(\varphi_n) \leq J_B(\phi_{+-}) + r_0.
$$
Then by~\cite[(ii) of Theorem 1]{Dancer1} we can take a (not relabeled) subsequence of $(n)$ and $\varphi \in \mathcal{S}_H(B) \cup \{0\}$ such that, for any $\vep > 0$, there exists $n_\vep \in \N$ satisfying
$$
J_B(\varphi) \leq J_B(\phi_{+-}) + r_0, \quad \|\varphi - \varphi_n\|_{H^1_0(\tilde{B})} < \vep
$$
for $n \in \N$ greater than $n_\vep$. One may rule out $\varphi = 0$. Indeed, if $\varphi = 0$, then $\varphi_n \to 0$ strongly in $H^1_0(\tilde{B})$ as $n \to +\infty$. On the other hand, we observe that
\begin{align*}
J_{\Omega_n}(\varphi_n) \geq \inf_{w \in \mathcal{S}(\Omega_n)} J_{\Omega_n}(w)
&= \frac{q-2}{2q} \Big[ \lambda_q C_q(\Omega_n)^q \Big]^{-2/(q-2)}\\
&\geq \frac{q-2}{2q} \left[ \lambda_q C_q(\tilde{B})^q \right]^{-2/(q-2)}
> 0, 
\end{align*}
where $C_q(\Omega_n)$ denotes the best constant of the Sobolev-Poincar\'e inequality \eqref{SI} with $\Omega = \Omega_n$ (see, e.g.,~\cite{Rabinowitz} and also~\cite[p.571]{AK13}). Here we also used the relation $C_q(\Omega_n) \leq C_q(\tilde{B})$. Hence it contradicts the fact that $J_{\Omega_n}(\varphi_n) = \frac{q-2}{2q} \|\nabla\varphi_n\|_{L^2(\Omega_n)}^2 \to 0$ as $n \to +\infty$. Thus we obtain $\varphi \neq 0$. Using \eqref{Sset0}, we can obtain either $\varphi = \phi_{+-}$ or $\varphi = - \phi_{+-}$. Hence $\varphi_n$ converges to either $\phi_{+-}$ or $-\phi_{+-}$ strongly in $H^1_0(\tilde{B})$ as $n \to +\infty$. However, due to~\cite[(ii) of Theorem 1]{Dancer1}, we infer that $\varphi_n$ coincides with either $\phi_n$ or $-\phi_n$, and this fact yields a contradiction to the assumption $\varphi_n \neq \pm \phi_n$. Thus \eqref{cl1} follows. Moreover, we can deduce that $J_{\Omega_n}$ is minimized over $\mathcal{S}_H(\Omega_n)$ by $\phi_n$ and $-\phi_n$ only. 

Finally, we shall prove that $\pm\phi_n$ also minimize $J_{\Omega_n}$ over $\mathcal{X}_H(\Omega_n)$. Let $v_{0,n} \in \mathcal{X}_H(\Omega_n)$ be such that $J_{\Omega_n}(v_{0,n}) \leq J_B(\phi_{+-}) + r_0$. Then the solution $v_n = v_n(x,s)$ to \eqref{eq:1.6}--\eqref{eq:1.8} with $\Omega = \Omega_n$ and $v_0 = v_{0,n}$ quasi-converges to a limit $\psi_n \in \mathcal{S}_H(\Omega_n)$ strongly in $H^1_0(\Omega_n)$ as $s \to +\infty$. Since the energy $s \mapsto J_{\Omega_n}(v_n(s))$ is non-increasing, it follows that
$$
J_{\Omega_n}(\psi_n) \leq J_{\Omega_n}(v_{0,n}) \leq J_B(\phi_{+-}) + r_0.
$$
By \eqref{cl1}, we obtain either $\psi_n = \phi_n$ or $\psi_n = -\phi_n$. Combining these facts, we deduce that $J_{\Omega_n}(\phi_n) \leq J_{\Omega_n}(v_{0,n})$. Hence $\pm \phi_n$ are minimizers of $J_{\Omega_n}$ over $\mathcal{X}_H(\Omega_n)$. Furthermore, if $v_{0,n} \in \mathcal{X}_H(\Omega_n)$ minimizes $J_{\Omega_n}$ over $\mathcal{X}_H(\Omega_n)$, that is, $J_{\Omega_n}(v_{0,n}) = J_{\Omega_n}(\phi_n)$, we obtain $v_{0,n} \in \mathcal{S}_H(\Omega_n)$. Indeed, we derive from \eqref{en} that
$$
c_q \int^s_0 \left\| \partial_s (|v_n|^{(q-2)/2}v_n)(s)\right\|_{L^2(\Omega_n)}^2 \, \d s + J_{\Omega_n}(v_n(s)) \leq J_{\Omega_n}(v_{0,n}),
$$
which along with the fact that $J_{\Omega_n}(\phi_n) = \inf_{w \in \mathcal{X}_H(\Omega_n)} J_{\Omega_n}(w)$ implies
$$
J_{\Omega_n}(v_n(s))  \equiv J_{\Omega_n}(v_{0,n}) \quad \mbox{and} \quad \partial_s (|v_n|^{(q-2)/2}v_n)(s) \equiv 0 \ \mbox{ a.e.~in } \Omega_n
$$
for $s \geq 0$. Hence $v_n(s) \equiv v_{0,n}$ and it solves \eqref{eq:1.10}, \eqref{eq:1.11} with $\Omega = \Omega_n$. Thus $v_{0,n}$ turns out to be an element of $\mathcal{S}_H(\Omega_n)$, and therefore, by \eqref{cl1}, $v_{0,n}$ coincides with either $\phi_n$ or $-\phi_n$. Consequently, we obtain $J_{\Omega_n}(w) > J_{\Omega_n}(\phi_n)$ if and only if
 $w \neq \pm\phi_n$.
\end{proof}

Now, we are ready to prove Theorem \ref{P:odd-stbl}, which can be proved along the same lines of Theorem 2 of~\cite{AK13} with the aid of lemmata proved so far. We provide here a proof for completeness.

\begin{proof}[Proof of Theorem \ref{P:odd-stbl}]
Since $\pm \phi_n$ are non-degenerate for $n \in \N$ large enough, they are isolated in $H^1_0(\Omega_n)$ from all the other non-trivial solutions for \eqref{eq:1.10}, \eqref{eq:1.11}. Hence let $r_n > 0$ be small enough that
\begin{equation}\label{111}
B_{\Omega_n}(\phi_n;r_n) \cap \mathcal{S}(\Omega_n) = \{\phi_n\},
\end{equation}
where $B_{\Omega_n}(\phi_n;r_n)$ denotes the ball in $H^1_0(\Omega_n)$ centered at $\phi_n$ with radius $r_n$. Let $\vep \in (0,r_n)$ be fixed. Then we claim that 
\begin{align}\label{cne}
c_{n,\vep} &:= \inf \left\{ J_{\Omega_n}(w) \colon w \in \mathcal{X}_H(\Omega_n), \ \|w - \phi_n\|_{H^1_0(\Omega_n)}  = \vep \right\} \nonumber\\
&> J_{\Omega_n}(\phi_n)
\end{align}
for $n \in \N$ large enough. Indeed, it has already been proved in Lemma \ref{l1} that $c_{n,\vep} \geq J_{\Omega_n}(\phi_n)$. Hence it suffices to show that $c_{n,\vep} \neq J_{\Omega_n}(\phi_n)$. Suppose to the contrary that $c_{n,\vep} = J_{\Omega_n}(\phi_n)$. Then there exists a sequence $(w_m)$ in $\mathcal{X}_H(\Omega_n)$ such that $J_{\Omega_n}(w_m) \to J_{\Omega_n}(\phi_n)$ and $\|w_m - \phi_n\|_{H^1_0(\Omega_n)} = \vep$. Hence we can extract a (not relabeled) subsequence of $(w_m)$ such that
$$
w_m \to \psi_n \quad \mbox{ weakly in } H^1_0(\Omega_n) \mbox{ and strongly in } L^q(\Omega_n)
$$
as $m \to +\infty$ for some $\psi_n \in H^1_0(\Omega_n)$. Since $\mathcal{X}_H(\Omega_n)$ is sequentially weakly closed in $H^1_0(\Omega_n)$, $\psi_n$ turns out to be an element of $\mathcal{X}_H(\Omega_n)$. It follows from Lemma \ref{l1} that $J_{\Omega_n}(\psi_n) \geq J_{\Omega_n}(\phi_n)$. Therefore we see that
\begin{align*}
 \frac 12 \|\nabla w_m\|_{L^2(\Omega_n)}^2
 &= J_{\Omega_n}(w_m) + \frac{\lambda_q}q \|w_m\|_{L^q(\Omega_n)}^q\\
 &\to J_{\Omega_n}(\phi_n) + \frac{\lambda_q}q \|\psi_n\|_{L^q(\Omega_n)}^q\\ 
 &\leq J_{\Omega_n}(\psi_n) + \frac{\lambda_q}q \|\psi_n\|_{L^q(\Omega_n)}^q
 = \frac 12 \|\nabla \psi_n\|_{L^2(\Omega_n)}^2.
\end{align*}
Thus we obtain
$$
w_m \to \psi_n \quad \mbox{ strongly in } H^1_0(\Omega_n)
$$
as $m \to +\infty$. Hence it follows that $J_{\Omega_n}(\psi_n) = J_{\Omega_n}(\phi_n)$ and $\|\psi_n - \phi_n\|_{H^1_0(\Omega_n)} = \vep \in (0,r_n)$; however, by virtue of Lemma \ref{l1}, they contradict each other. Thus we conclude that $c_{n,\vep} > J_{\Omega_n}(\phi_n)$.

Since $J_{\Omega_n}(\cdot)$ is continuous in $H^1_0(\Omega_n)$, one can take $\delta_{n,\vep} \in (0,\vep)$ such that
$$
J_{\Omega_n}(v_{0,n}) < c_{n,\vep}
$$
for any $v_{0,n} \in \mathcal{X}_H(\Omega_n)$ satisfying $\|v_{0,n}-\phi_n\|_{H^1_0(\Omega_n)} < \delta_{n,\vep}$. Hence let $v_{0,n} \in \mathcal{X}_H(\Omega_n)$ satisfy $\|v_{0,n}-\phi_n\|_{H^1_0(\Omega_n)} < \delta_{n,\vep}$ and let $v_n = v_n(x,s)$ be the weak solution to \eqref{eq:1.6}--\eqref{eq:1.8} with $\Omega = \Omega_n$ and $v_0 = v_{0,n}$. Since $s \mapsto J_{\Omega_n}(v_n(s))$ is non-increasing, we have 
$$
J_{\Omega_n}(v_n(s)) \leq J_{\Omega_n}(v_{0,n}) < c_{n,\vep}
$$
for any $s \geq 0$. Therefore, by virtue of \eqref{cne}, $v_n(s)$ cannot go beyond the boundary of the ball $B_{\Omega_n}(\phi_n;\vep)$ for any $s \geq 0$, that is, it holds that 
\begin{equation}\label{ep-nbhd}
\sup_{s \geq 0} \|v_n(s)-\phi_n\|_{H^1_0(\Omega_n)} \leq \vep
\end{equation}
(cf.~\cite[Proof of Theorem 3]{A16}). Thus $\phi_n$ turns out to be stable under the dynamical system in $\mathcal{X}_H(\Omega_n)$ generated by \eqref{eq:1.6}--\eqref{eq:1.8} with $\Omega = \Omega_n$. 

Furthermore, since each solution $v_n(s)$ of \eqref{eq:1.6}--\eqref{eq:1.8} with $\Omega = \Omega_n$ and $v_0 = v_{0,n} \in \mathcal{X}_H(\Omega_n)$ quasi-converges to an element of $\mathcal{S}_H(\Omega_n)$ strongly in $H^1_0(\Omega_n)$ as $s \to +\infty$ and $\phi_n$ is isolated in $H^1_0(\Omega_n)$ from all the other elements of $\mathcal{S}_H(\Omega_n)$ (see~\eqref{111}), we deduce from the stability of $\phi_n$ that $v_n(s) \to \phi_n$ strongly in $H^1_0(\Omega_n)$ as $s \to +\infty$, provided that $v_{0,n} \in \mathcal{X}_H(\Omega_n)$ and $\|v_{0,n} -\phi_n\|_{H^1_0(\Omega_n)}$ is small enough. Finally, the exponential stability follows from Theorem \ref{T:sc-conv}. This completes the proof.
\end{proof}

\begin{remark}[Positive and even asymptotic profiles in dumbbell domains]
The above argument can also be applied to positive and even (with respect to the hyperplane $H$) solutions on dumbbell domains with thin channels by replacing odd functions with even ones. 
\end{remark}

\section*{Acknowledgments}
The author wishes to thank anonymous referees for their careful reading and fruitful comments to improve the readability of the manuscript. The last section is inspired by referees' questions. The author is supported by JSPS KAKENHI Grant Numbers JP21KK0044, JP21K18581, JP20H01812, JP18K18715 and JP20H00117, JP17H01095. This work was also supported by the Research Institute for Mathematical Sciences, an International Joint Usage/Research Center located in Kyoto University.

\appendix

\section{Taylor's theorem}\label{S:Taylor}

In this section, we shall recall the well-known mean-value theorem as well as Taylor's theorem for operators in Banach spaces for the convenience of the reader. We refer the reader to, e.g.,~\cite{Zeidler,Willem} for details on Fr\'echet and G\^ateaux derivatives of operators defined on Banach spaces (see also Notation in $\S \ref{S:Intro}$). Let us start with the mean-value theorem.

\begin{theorem}[Mean-value theorem for operators]\label{T:MVT}
Let $x,y \in X$ and let $I = [x,y] = \{(1-\theta)x+\theta y \colon \theta \in [0,1]\}$. Let $U$ be an open set in $X$ such that $I \subset U$ and let $F : U \subset X \to Y$ be G\^ateaux differentiable on $I$ such that the G\^ateaux derivative $\D_G F : I \subset X \to \mathscr{L}(X,Y)$ of $F$ is bounded in $\mathscr{L}(X,Y)$ on $I$. Then it holds that
$$
\|F(y)-F(x)\|_Y \leq \sup_{\theta \in [0,1]} \left\| \D_G F((1-\theta)x+\theta y) \right\|_{\mathscr{L}(X,Y)} \|y-x\|_X.
$$
\end{theorem}

\begin{proof}
Let $\eta \in Y^*$ be such that $\|\eta\|_{Y^*} = 1$ and $\langle \eta , F(y) - F(x) \rangle_Y = \|F(y)-F(x)\|_Y$ (indeed, such an $\eta$ exists thanks to Hahn-Banach's theorem; see, e.g.,~\cite[Corollary 1.3]{B-FA}). Since $F$ is G\^ateaux differentiable on $I$, we see that $\theta \mapsto \varphi(\theta) := \langle \eta, F((1-\theta)x + \theta y)\rangle_Y$ is differentiable on $[0,1]$. Hence using the standard mean-value theorem, we can take $\theta_0 \in (0,1)$ such that $\varphi(1)-\varphi(0) = \varphi'(\theta_0) (1-0)$, that is,
\begin{align*}
\|F(y)-F(x)\|_Y &= \langle \eta, [\D_G F((1-\theta_0)x + \theta_0 y)](y-x) \rangle_Y\\
&\leq \|[\D_G F((1-\theta_0)x + \theta_0 y)](y-x)\|_Y\\
&\leq \|\D_G F((1-\theta_0)x+\theta_0 y)\|_{\mathscr{L}(X,Y)} \|y-x\|_X\\
&\leq \sup_{\theta \in [0,1]}\|\D_G F((1-\theta)x+\theta y)\|_{\mathscr{L}(X,Y)} \|y-x\|_X.
\end{align*}
This completes the proof.
\end{proof}

Here and henceforth, for each $j \in \N$, $T \in \mathscr{L}^{(j)}(X,Y)$ and $x \in X$, we shall simply write $T(x,x,\ldots,x) = T x^j$. 

\begin{theorem}[Taylor's theorem for operators]\label{T:Taylor}
Let $x,y \in X$ and let $I = [x,y] = \{(1-\theta)x+\theta y \colon \theta \in [0,1]\}$. Let $U$ be an open set in $X$ such that $I \subset U$ and let $F : U \subset X \to Y$ be $(n-1)$-times Fr\'echet differentiable in $U$ such that the $(n-1)$-th Fr\'echet derivative $F^{(n-1)} : U \subset X \to \mathscr{L}^{(n-1)}(X,Y)$ of $F$ is G\^ateaux differentiable on $I$ and the G\^ateaux derivative $\D_G F^{(n-1)}$ of $F^{(n-1)}$ is bounded in $\mathscr{L}^{(n)}(X,Y)$ on $I$. Then it holds that
\begin{equation}\label{taylor}
 F(y) = F(x) + \sum_{j=1}^{n-1} \dfrac{[F^{(j)}(x)]}{j!}(x-y)^j + e,
\end{equation}
where $e \in Y$ satisfies
$$
\|e\|_Y \leq \sup_{\theta \in [0,1]} \|\D_G F^{(n-1)}((1-\theta)x+\theta y)\|_{\mathscr{L}^{(n)}(X,Y)} \|y-x\|_X^n.
$$
\end{theorem}

\begin{proof}
Set
$$
P_{n-1}(w) = \sum_{j=0}^{n-1} \frac{[F^{(j)}(x)]}{j!}(w-x)^j
$$
and
$$
G(w) = F(w) - P_{n-1}(w)
$$
for $w \in U$. Then $G$ is $(n-1)$-times Fr\'echet differentiable on $U$ such that
\begin{equation}\label{MO}
G^{(\ell)}(x) = 0 \quad \mbox{ for } \ \ell = 0,1,\ldots,n-1.
\end{equation}
Moreover, by assumption, $G^{(n-1)}$ is G\^ateaux differentiable on $I$ and $\sup_{t \in [0,1]} \| \D_G G^{(n-1)}([x,y]_t)\|_{\mathscr{L}^{(n)}(X,Y)} < +\infty$. In what follows, we write $[x,y]_t = (1-t)x + ty$ and note that $[x,[x,y]_t]_s = [x,y]_{st}$ for $s,t \in [0,1]$. Moreover, using \eqref{MO} and Theorem \ref{T:MVT} repeatedly, we see that
\begin{align}
 \|G(y)\|_Y
 &= \|G(y) - G(x)\|_Y\nonumber\\
 &\leq \sup_{t_1 \in [0,1]} \|G'([x,y]_{t_1})\|_{\mathscr{L}(X,Y)} \|y-x\|_X\nonumber\\
 &= \sup_{t_1 \in [0,1]} \|G'([x,y]_{t_1}) - G'(x)\|_{\mathscr{L}(X,Y)} \|y-x\|_X\nonumber\\
 &\leq \sup_{t_1 \in [0,1]} \sup_{t_2 \in [0,1]} \|G''([x,[x,y]_{t_1}]_{t_2}) \|_{\mathscr{L}^{(2)}(X,Y)} t_1 \|y-x\|_X^2\nonumber\\
 &\leq \sup_{t \in [0,1]} \|G''([x,y]_t) \|_{\mathscr{L}^{(2)}(X,Y)} \|y-x\|_X^2\nonumber\\
 &\leq \sup_{t \in [0,1]} \|G^{(n-1)}([x,y]_t) \|_{\mathscr{L}^{(n-1)}(X,Y)} \|y-x\|_X^{n-1}\label{(n-1)}\\
 &\leq \sup_{t \in [0,1]} \| \D_G G^{(n-1)}([x,y]_t)\|_{\mathscr{L}^{(n)}(X,Y)} \|y-x\|_X^n,\nonumber
\end{align}
which ensures the desired assertion for $e = G(y)$. This completes the proof.
\end{proof}

\begin{remark}\label{R:Taylor}
If $F : U \subset X \to Y$ is only of class $C^{n-1}$ in $U$ in the sense of Fr\'echet derivative, then we can still obtain \eqref{taylor} along with $e \in Y$ satisfying only
$$
\lim_{\|y-x\|_X \to 0} \frac{\|e\|_Y}{\|y-x\|_X^{n-1}} = 0.
$$
Indeed, as in \eqref{(n-1)}, we can derive from the continuity of $G^{(n-1)}$ that
\begin{align*}
\frac{\|G(y)\|_Y}{\|y-x\|_X^{n-1}}
 &\leq \sup_{t \in [0,1]} \|G^{(n-1)}([x,y]_t) - G^{(n-1)}(x)\|_{\mathscr{L}^{(n-1)}(X,Y)} \\
&\to 0 \quad \mbox{ as } \ \|x-y\|_X \to 0.
\end{align*}
Setting $e = G(y)$, we obtain the desired conclusion.
\end{remark}

Finally, we shall give a proof for the fact that $J$ is of class $C^3$ in $H^1_0(\Omega)$, provided that $q \geq 3$. Let $w \in H^1_0(\Omega)$ be arbitrarily fixed. It is well known that $J$ is of class $C^2$ in $H^1_0(\Omega)$ and its second Fr\'echet derivative $J''(w) \in \mathscr{L}(H^1_0(\Omega),H^{-1}(\Omega)) = \mathscr{L}^{(2)}(H^1_0(\Omega),\R)$ at $w$ is represented by
$$
[J''(w)](u, v) = \int_\Omega \nabla u \cdot \nabla v \, \d x - \lambda_q (q-1) \int_\Omega |w|^{q-2} uv \, \d x
$$
for $u,v \in H^1_0(\Omega)$ (see, e.g.,~\cite[Corollary 1.13]{Willem}). Therefore since $q \geq 3$, we can see that
\begin{align*}
\MoveEqLeft{
\left\| \frac{J''(w+ he) - J''(w)}{h} + \lambda_q (q-1)(q-2) |w|^{q-4}w e \right\|_{\mathscr{L}^{(2)}(H^1_0(\Omega),\R)}
}\\
&\leq C \lambda_q (q-1) \left\| \frac{|w+he|^{q-2}-|w|^{q-2}}h - (q-2) |w|^{q-4}w e \right\|_{L^{q/(q-2)}(\Omega)}\\
&\to 0 \quad \mbox{ as } \ h \to 0
\end{align*}
for $e \in H^1_0(\Omega)$. Thus $J'' : H^1_0(\Omega) \to \mathscr{L}^{(2)}(H^1_0(\Omega);\R)$ is G\^ateaux differentiable at $w$ and its derivative $\D_G J''(w) \in \mathscr{L}^{(3)}(H^1_0(\Omega);\R)$ at $w$ is represented as
$$
[\D_G J''(w)](e,u,v) = \lambda_q (q-1)(q-2) \int_\Omega |w|^{q-4}w e u v \, \d x
$$
for $e,u,v \in H^1_0(\Omega)$. Moreover, one can check from $q \geq 3$ that $w \mapsto \D_G J''(w)$ is a continuous map from $H^1_0(\Omega)$ into $\mathscr{L}^{(3)}(H^1_0(\Omega),\R)$, and therefore, $J''$ also turns out to be Fr\'echet differentiable at $w$ and its Fr\'echet derivative $J^{(3)}(w)$ at $w$ coincides with $\D_G J''(w)$.

\section{Fundamental inequalities}\label{S:fe}

We first prove \eqref{fundame}. We can assume that $a > b$ and $a,b \neq 0$ without  loss of generality. We see that
 \begin{align*}
 \left(
 |a|^{q-2}a - |b|^{q-2}b
 \right) (a - b)
 &= \left(
 \int^a_b \frac{\d}{\d\xi} (|\xi|^{q-2} \xi) \, \d\xi
 \right)
 \left(
 \int^a_b 1 
\, \d\xi
 \right)
 \\
 &\geq \left[
 \int^a_b 
 \left(
 \frac{\d}{\d\xi} (|\xi|^{q-2} \xi)
 \right)^{1/2}
  1^{1/2}
 \,\d\xi
 \right]^2\\
  &= (q-1)\left[
  \int^a_b 
  |\xi|^{\frac{q-2}2}
 \d\xi
 \right]^2\\
  &= (q-1)\left[
  \int^a_b 
  \frac 2 q \frac{\d}{\d \xi} (|\xi|^{\frac{q-2}2}\xi)
\, \d\xi
 \right]^2\\
&= \frac{4(q-1)}{q^2} \left| |a|^{\frac{q-2}2}a - |b|^{\frac{q-2}2}b \right|^2.
\end{align*}

Inequality \eqref{fun_ineq} is standard. We next prove \eqref{sublin-ineq2}. In case $a$ and $b$ have the same sign, \eqref{fun_ineq} is applicable. In case $a$ and $b$ have different signs, we may simply assume $a > 0 > b$. Set $c = -b > 0$. Since $p \in (0,1)$, we see that
$$
\frac{a^p+c^p}2 \leq \left( \frac{a+c}2\right)^p,
$$
which yields
$$
a^p + c^p \leq 2^{1-p} (a+c)^p.
$$
It follows that
\begin{align*}
 \frac{a^p+c^p}{a+c} \leq 2^{1-p} (a+c)^{p-1} \leq 2^{1-p} a^{p-1}.
\end{align*}
Thus we have \eqref{sublin-ineq2}.

\end{document}